\documentclass[12pt]{article}

\usepackage[english]{babel}
\usepackage[utf8]{inputenc}
\usepackage{amsmath,amsfonts,amssymb}\allowdisplaybreaks
\usepackage[colorlinks]{hyperref}
\usepackage{color,colortbl}
\usepackage{graphicx}
\usepackage{geometry}
\usepackage{caption,subcaption}
\usepackage{multicol}
\usepackage{listing}
\usepackage{booktabs}
\usepackage{array}
\usepackage{setspace}
\usepackage{pdfpages}
\usepackage{textcomp}
\usepackage[numbers]{natbib}
\usepackage{graphicx}
\usepackage{rotating}
\usepackage{tikz}
\usepackage{xfrac}
\usepackage{enumerate,enumitem}
\usepackage{tabu}
\usepackage{nicefrac}
\usepackage{bbm}
\usepackage{titlesec}
\usepackage{cancel}
\usepackage{amsthm}
\usepackage{amsthm}
\usepackage{sectsty}
\usepackage{mathtools}

\hypersetup{
	linkcolor={black},
	urlcolor={darkblue},
	citecolor = {black}
}

\newcommand{\dup}{\textup{d}}

\newcommand{\Z}{\mathbb{Z}}

\newcommand{\R}{\mathbb{R}}

\newcommand{\s}{\mathbb{S}}

\newcommand{\ov}[1]{\overline{#1}}

\sectionfont{\centering\large}
\subsectionfont{\center\normalsize}

\title{\scalebox{0.75}{Topology and Embeddedness of Lawson's Bipolar Surfaces in the $5$-Sphere}}
\date{}
\author{\scalebox{0.8}{Elena Mäder-Baumdicker, Melanie Rothe}}

\newtheorem{theorem}{Theorem}[section] 
\newtheorem{lemma}[theorem]{Lemma}
\newtheorem{proposition}[theorem]{Proposition}
\newtheorem{corollary}[theorem]{Corollary}

\theoremstyle{definition}
\newtheorem{definition}[theorem]{Definition}
\newtheorem{example}[theorem]{Example}
\newtheorem{remark}[theorem]{Remark}
\newtheorem{construction}[theorem]{Construction}

\newtheorem{maintheorem}{Theorem}

\begin{document}
\maketitle

\begin{abstract}
\noindent
We give a topological classification of Lawson's bipolar minimal surfaces corresponding to his $\xi$- and $\eta$-family. Therefrom we deduce upper as well as lower bounds on the area of these surfaces, and find that they are not embedded.
\end{abstract}


\section{Introduction}
\renewcommand{\thefootnote}{\fnsymbol{footnote}} 
\footnotetext{\textbf{MSC2020:} 53A05 $\bullet$ 53A10 $\bullet$ 53C42
\textbf{Key words:} immersed surfaces $\bullet$ minimal surfaces $\bullet$ minimal surfaces and constant mean curvature surfaces in \(S^3\) $\bullet$ Lawson surfaces $\bullet$ spherical geometry}

Minimal surface theory is a classical and still very active field of research in Differential Geometry and Geometric Analysis. As the critical points of the area functional, minimal surfaces necessarily satisfy a curvature condition. Therefore, they constitute exceptional, two-dimensional submanifolds of a Riemannian manifold of dimension $\ge 3$. In this sense, one particular motivation behind minimal surface theory is to obtain a better geometric understanding of the potential ambient spaces.

Unlike Euclidean space $\R^n$ as an ambient manifold, the $n$-sphere $\s^n$ allows for closed minimal surfaces. So, a very general interest lies in their topological classification and how this enters their geometry. In this context, various existence results are known for $\s^3$, e.g.\ the surfaces from \cite{Lawson}, \cite{KPS}, \cite{KY} and \cite{CS} (for an introduction see \citep{Brendle}). It must, however, be added that regarding each topological class, the list of known examples is limited. Compared to the results in $\s^3$, the setting of higher codimension is rather unexplored but, at the same time, of strong interest. The reason is that closed minimal surfaces in $\s^n$, also in higher codimension, turned out to be relevant in further prominent geometric variational problems, in particular with respect to their topological class. On the one hand, this refers to the search for minimizers of the Willmore energy, which is the natural energy to employ when asking for optimally shaped, immersed closed surfaces in $\R^n$. The point is that closed minimal surfaces in $\s^n$ stereographically project onto critical surfaces in $\R^n$ (called Willmore surfaces), where their area is mapped to the Willmore energy of the stereographic projection. Therefore, finding closed minimal surfaces in $\s^n$ and estimating their area yields comparison surfaces for the Willmore problem (cf.\ \cite{Kusner}, \cite{MBH}). On the other hand, closed minimal surfaces in $\s^n$ play an important role in spectral geometry by providing extremal metrics for the normalized eigenvalue functionals of the Laplace-Beltrami operator on closed 2-manifolds (cf.\ \cite{Takahashi}, \cite{JNP}, \cite{Lapointe}, \cite{Penskoi}). 

In all these regards, Lawson's closed minimal surfaces in $\s^3$ have certainly played an important role. After initially the only known closed minimal surfaces in $\s^3$ were great two-spheres and the Clifford torus, a large number of less obvious examples was provided by H.\ B.\ Lawson in 1970. In \cite{Lawson}, he first deduced a general existence theorem relying on successive application of the Schwarz reflection principle and applied it in order to construct his well-known three families of closed minimal surfaces $(\xi_{m,k})$, $(\tau_{m,k})$ and $(\eta_{m,k})$. These families contain examples for every topological class, particularly embedded examples for every orientable genus. Concerning the Willmore problem, we remark that the stereographic projection of the Clifford torus, realized by both $\xi_{1,1}$ and $\tau_{1,1}$, is the minimizer of the Willmore energy among tori (cf.\ \cite{WillmoreConjecture}). More generally, the surfaces $\xi_{g,1}$, $g\ge 2$, are conjectured to be the minimizers for orientable genus $g$ (cf.\ \cite{Kusner}).

Besides the above construction procedure, Lawson provided in \cite{Lawson} a detailed investigation of the minimal surface equation in $\s^3$, which, inter alia, led to the notion of the so-called polar variety as well as the bipolar surface associated to each minimal surface in $\mathbb{S}^3$. More precisely, given a conformal immersion $\psi\colon\Sigma\to \mathbb{S}^3\subseteq\R^4$ of a closed, orientable, two-dimensional manifold $\Sigma$ (perhaps arising as double-cover of a non-orientable manifold), then each choice $\nu\colon\Sigma\to \mathbb{S}^3$ of a unit normal is again minimal, but possibly has branch points. In turn, the bipolar surface $\widetilde{\psi}:=\psi\wedge \nu\colon\Sigma\to \s^5\subseteq \R^6$ is again a conformal minimal immersion and, remarkably, is conformal to $\psi$ itself. 

With the example of the bipolar Lawson surfaces $\widetilde{\tau}_{m,k}$, H.\ Lapointe showed in \cite{Lapointe} that various properties of the bipolar surface can crucially differ from the original surface. Firstly, this concerns the topology: For example, it is known (Theorem 1.3.1 in \cite{Lapointe}) that if $mk\equiv 3\textup{ mod } 4$, then $\tau_{m,k}$ is a torus in $\s^3$, but $\widetilde{\tau}_{m,k}$ is a Klein bottle in $\s^5$. In that case, the immersion $\widetilde{\psi}$ is already well-defined on a quotient of $\Sigma$ obtained by a covering map of degree two with an orientation-reversing deck-transformation. A further point of comparison is the embeddedness: The example of the surface $\widetilde{\tau}_{3,1}$ shows the bipolar surface can turn out to be embedded, even if the original surface was not (cf.\ \cite{MBH}). In the present article, we show that also the converse does occur (cf. Corollary \ref{embeddedness}). 

From the perspective of the Willmore problem, we furthermore note that the stereographic projection of the bipolar surface $\widetilde{\tau}_{3,1}$, a Klein bottle in $\s^4\subseteq\s^5$, was, among Klein bottles in $\R^4$, identified as the unique minimizer in its conformal class, and is conjectured to be the minimizer in its topological class (cf.\ \cite{MBH}).

The motivation for the present paper is to study the other two families of bipolar Lawson surfaces $\big(\widetilde{\xi}_{m,k}\big)$ and $\big(\widetilde{\eta}_{m,k}\big)$, concerning topology, embeddedness and area estimates. Note that by the former we refer to the topology of the smallest possible domain of the immersion $\widetilde{\psi}$ which is, perhaps, realized by a quotient of $\Sigma$ under a covering map. Furthermore, note that, by the Li-Yau inequality (cf.\ \cite{LiYau}), lower bounds on the area of a closed minimal surface in $\s^n$ can be obtained from an investigation of the embeddedness. In particular, the area (and thus the Willmore energy) is at least $8\pi$ if the surface has self-intersections. 

Our approach uses the basic data of the Schwarz reflection process from the construction procedure in \cite{Lawson}, arising from a geodesic polygon $\Gamma\subseteq\s^3$. In particular, we employ the algebraic properties of the corresponding group $G$ generated by Schwarz reflections. Shortly speaking, we detect whether different $G$-copies of the initial piece of surface in $\s^3$ are mapped to the same pieces in the bipolar surface. Thereby, we can estimate how often the bipolar surface is covered by $\Sigma$. Within this setting, the occurrence of a higher cover is related to a certain element of $G$, and we find that the behavior of the orientation under such covering map can be tracked in terms of its (purely algebraic) properties within the group $G$. In the special case of the surfaces $\xi_{m,k}$ and $\eta_{m,k}$, we are sure that this method already provides a full characterization of the topology of the corresponding bipolar surfaces. By counting covers and computing multiplicities at certain points in the bipolar surfaces, we can also determine area bounds, and, by considering the tangent spaces at points of higher multiplicity, analyze their embeddedness. 

We remark that similar results on the bipolar $\tau$-family in \cite{Lapointe} heavily rely on the knowledge of explicit parametrizations, whereas such parametrizations are not known for the $\xi$- and $\eta$-family. 

Finally, we arrive at the following theorem on the family $(\widetilde{\xi}_{m,k})$. Note that, for a nicer presentation of formulas, we shifted the indices.
\begin{maintheorem}
Let $m,\,k\in \Z_{\ge 2}$ such that $m>2$ or $k>2$. Then, the minimal bipolar surface $\tilde{\xi}_{m-1,k-1}:\Sigma \to \mathbb{S}^5$ is orientable. Moreover,
\begin{enumerate}[label=\normalfont{(\roman*)}]
\item if both $m$ and $k$ are even, the Euler characteristic is 
\begin{align*}
\chi\Big(\tilde{\xi}_{m-1,k-1}\Big)=1-(m-1)(k-1)
\end{align*}
and we have
\begin{align*}
2\pi \max\{m, k\}\le\textup{area}\Big(\widetilde{\xi}_{m-1,k-1}\Big)<2\pi(mk+k-m)\,;
\end{align*}
\item if $m$ or $k$ is odd, the Euler characteristic is 
\begin{align*}
\chi\Big(\tilde{\xi}_{m-1,k-1}\Big)=2\big(1-(m-1)(k-1)\big)
\end{align*}
and we have
\begin{align*}
4\pi \max\{m, k\}\le\textup{area}\Big(\widetilde{\xi}_{m-1,k-1}\Big)<4\pi(mk+k-m)\,.
\end{align*}
\end{enumerate}
\end{maintheorem}

\noindent For the family $(\widetilde{\eta}_{m,k})$ we show the following theorem.
\begin{maintheorem}
Let $m,\,k\in \Z_{\ge 2}$ such that $m>2$ or $k>2$. Then, the minimal bipolar surface $\tilde{\eta}_{m-1,k-1}:\Sigma \to \mathbb{S}^5$ is orientable. Moreover,
\begin{enumerate}[label=\normalfont{(\roman*)}]
\item if both $m$ and $k$ are even, the Euler characteristic is 
\begin{align*}
\chi\Big(\tilde{\eta}_{m-1,k-1}\Big)=1-(m-1)(k-1)
\end{align*}
and we have
\begin{align*}
2\pi \max\{m, k\}\le\textup{area}\Big(\widetilde{\eta}_{m-1,k-1}\Big)<2\pi(3mk -3k - m)\,;
\end{align*}

\item if $m$ or $k$ is odd, the Euler characteristic is 
\begin{align*}
\chi\Big(\tilde{\eta}_{m-1,k-1}\Big)= 2\big(1-(m-1)(k-1)\big)
\end{align*}
and we have
\begin{align*}
4\pi \max\{m, k\}\le\textup{area}\Big(\widetilde{\eta}_{m-1,k-1}\Big)<4\pi(3mk -3k - m)\,.
\end{align*}

\end{enumerate}
\end{maintheorem}

Note that, for technical reasons, the above theorems do not include the bipolar surfaces of the Clifford torus $\xi_{1,1}$ and the Klein bottle $\eta_{1,1}$. However, these particular surfaces were already treated in \cite{Lapointe} as they coincide with $\tau_{1,1}$ and $\tau_{2,1}$.

Concerning the strategy of the proofs, we first translate Lawson's construction procedure from \cite{Lawson} of a closed minimal surface in $\s^3$ into a corresponding immersion $\psi\colon S\to \s^3$, defined on a smallest fundamental domain $S$. Then, using $S$ or possibly its orientable double cover $\widetilde{S}$, we continue with the definition of an associated Gauss map $\nu$, and therewith introduce the corresponding bipolar immersion $\widetilde{\psi}$. The crucial point about our setup is that the symmetries in the images of the considered maps directly relate to smooth self-mappings of $S$. From that, we can on the one hand decide if these contain smooth covering maps on $S$ under which $\widetilde{\psi}$ is invariant, and on the other hand, if such maps are orientation-reversing or not. Furthermore, by studying these self-mappings, we can detect points of higher multiplicity in the bipolar surface and the different tangent planes at such points. In the end, a detailed analysis of the respective group generated by Schwarz reflections allows to apply the above methods to the bipolar Lawson surfaces $\widetilde{\xi}_{m,k}$ as well as $\widetilde{\eta}_{m,k}$, which finally lead to our main theorems.

The paper is organized as follows: We start with a preliminary section, Section \ref{preparations}, intended to provide the setting for Lawson's description of the bipolar surface. In Section \ref{preimage_construction}, we treat the immersions, their Gauss maps and, finally, the corresponding bipolar immersions. In Section \ref{xi_eta_surfaces} we apply the methods from Section \ref{preimage_construction} to the bipolar Lawson surfaces $\widetilde{\xi}_{m,k}$ and $\widetilde{\eta}_{m,k}$.

\subsection*{Acknowledgments}
We would like to thank Karsten Große-Brauckmann, Jonas Hirsch and Francisco Torralbo for discussions and their interest in our work. Moreover, we thank Sukie Vetter and the geometry group at the Department of Mathematics of the TU Darmstadt for their helpful comments. The first author is supported by the DFG (MA 7559/1-2) and thanks the DFG for the support. We furthermore remark that the results of this article are part of the second author's PhD thesis.

\section{Preparations}\label{preparations}
\subsection{Minimal Submanifolds in the $n$-Sphere}
In an $\ov{m}$-dimensional Riemannian manifold $\big(\ov{M},\langle\cdot,\cdot\rangle\big)$, consider an $m$-dimensional, possibly immersed submanifold $M\subseteq \ov{M}$, equipped with the metric $g:=\iota^*\langle\cdot,\cdot\rangle$ induced by the canonical inclusion $\iota: M\hookrightarrow \ov{M}$. 
At each point $P\in M\subseteq \ov{M}$, the tangent space $T_P\ov{M}$ splits into 
\begin{align*}
T_P\ov{M}=T_P M \oplus N_P M\cong \R^m\oplus \R^{\ov{m}-m},
\end{align*}
where 
\begin{align*}
N_P M:=\Big(T_P M\Big)^{\perp,\langle\cdot,\cdot\rangle}
\end{align*}
is called \textit{the normal space to $M$ at $P$}. We denote the corresponding orthogonal projections, depending smoothly on $P$, by
\begin{alignat*}{2}
&(\cdot)^T|_P &&\colon T_P\ov{M}\to T_P M\,,\\[5pt]
&(\cdot)^N|_P &&\colon T_P\ov{M}\to N_P M\,.
\end{alignat*}
Let now $\ov{\nabla}$ be the Levi-Civita connection on $\ov{M}$ and $X_1,...,X_m\in\mathfrak{X}(M)$ be smooth vector fields on $M$, which, as such, locally extend to smooth vector fields on $\ov{M}$. With respect to the decomposition from above, we have
\begin{align}
\ov{\nabla}_X Y=\big(\ov{\nabla}_X Y\big)^T+\big(\ov{\nabla}_X Y\big)^N.\label{LeviCivitaDecomposition}
\end{align}
\begin{definition}
The normal component of \eqref{LeviCivitaDecomposition},
\begin{align*}
B\colon\mathrm{X}(M)\times \mathrm{X}(M)\to \Gamma(NM),~B(X,Y):=\big(\ov{\nabla}_X Y\big)^N\,, 
\end{align*}
where $\Gamma(NM)$ denotes the set of smooth sections in the normal bundle $NM$, is called the \textit{second fundamental form $B$ of the Riemannian submanifold $M\subseteq \ov{M}$}. 
\end{definition}
\noindent From the properties of $\ov{\nabla}$ it follows that $B$ is $C^\infty(M)$-bilinear and symmetric. Moreover, for any $X,\,Y\in\mathfrak{X}(M)$, $B(X,Y)$ is independent of the local extensions of $X,\,Y$ to smooth vector fields on $\ov{M}$ and the value of $B(X,Y)|_P$ at $P\in M$ does only depend on the values $X|_P$ and $Y|_P$. The Gauss formula
\begin{align*}
\ov{\nabla}_X Y-\nabla_X Y=B(X,Y)\,,
\end{align*}
where $\nabla$ is the Levi-Civita connection on $(M,g)$, shows that the second fundamental form precisely describes the difference between the interior geometry of $M$ and its exterior geometry in the ambient manifold $\ov{M}$. In this context, we want to include two well-known equations.
\begin{definition}
To every normal vector field $N\in\Gamma(NM)$ we associate the map
\begin{align*}
\Pi_N\colon\mathfrak{X}(M)\times\mathfrak{X}(M)\to C^\infty(M),\quad \Pi_N(X,Y):=\langle N, B(X,Y)\rangle\,,
\end{align*}
defining a symmetric bilinear form at each point. The corresponding, pointwise self-adjoint linear map $W_N:\mathfrak{X}(M)\to\mathfrak{X}(M)$, i.e.,
\begin{align*}
\langle W_N(X),Y\rangle=\Pi_N(X,Y)\quad\textup{for }X,\,Y\in\mathfrak{X}(M)\,,
\end{align*}
is called the \textit{Weingarten map in the direction of} $N$.
\end{definition}
\begin{proposition}[The Weingarten Equation]\label{Weingarten_equation}
Let $X\in\mathfrak{X}(M)$ and $N\in\Gamma(NM)$. Then, we have
\begin{align*}
\Big(\ov{\nabla}_X N\Big)^T=-W_N(X)\,,
\end{align*}
where $N$ is locally extended to an open subset of $\ov{M}$.
\end{proposition}
\noindent The second equation shows that, more detailed, the second fundamental form encodes the difference between the curvature tensors $Rm$ and $\ov{Rm}$ of $M$ and $\ov{M}$.
\begin{proposition}[The Gauss Equation]\label{Gauss_equation}
Let $W,\, X,\, Y,\, Z\in\mathfrak{X}(M)$. Then, we have
\begin{align*}
\ov{Rm}(W,X,Y,Z)=Rm(W,X,Y,Z)-\langle B(W,Z),B(X,Y)\rangle+\langle B(W,Y),B(X,Z)\rangle\,.
\end{align*}
\end{proposition}
\noindent To continue, the second fundamental form allows to introduce one of the simplest and perhaps most relevant geometric invariants of the Riemannian submanifold $M\subseteq\ov{M}$.
\begin{definition}
The \textit{mean curvature vector} of $M$ is defined by
\begin{align*}
H:=\textup{tr}_g(B)\,,
\end{align*}
i.e., in terms of a local, $g$-orthonormal frame $(X_1,...,X_m)$ on $M$, we have
\begin{align*}
H=\sum_{i=1}^m B(X_i,X_i)\,.
\end{align*}
\end{definition}
\noindent Towards an interpretation of the mean curvature vector, suppose that $M=M_\psi$ is described as the image of a smooth immersion $\psi\colon \Sigma\to M$ of a smooth, $m$-dimensional manifold $\Sigma$. In this setting, we have $g=\psi^*\langle\cdot,\cdot\rangle$. Then, the \textit{area} of $\psi$ is defined by 
\begin{align*}
\textup{area}(\psi):=\int\limits_\Sigma \dup\mu_g\,.
\end{align*}
Consider a \textit{smooth variation of $\psi$}, that is, a smooth family $\psi\colon (-1,1)\times \Sigma\to \ov{M}$ of immersions such that $\Psi(0,\cdot)=\psi$ and $\Psi(t,\cdot)|_{\partial \Sigma}=\psi|_{\partial \Sigma}$ for all $t\in(\text{-}1,1)$. Let furthermore $\partial_t$ be the canonical vector field along $(-1,1)$ and set $E:=\Psi_* \partial_t |_{t=0}$.  Then, as for example deduced in \cite{Lawson_Lecture}, we have
\begin{align*}
\frac{\dup }{\dup t}\,\textup{area}\big(\Psi(t,\cdot)\big)\Bigg|_{t=0}=-\int\limits_\Sigma \langle H,E\rangle~ \dup \mu_{g}\,.
\end{align*}
This means the mean curvature $H$ of $M_\psi$ is exactly the gradient of the area functional on the space of immersions that describe the submanifold $M\subset\ov{M}$, i.e., deformations of $\psi$ in the direction of $H$ provide the fastest decrease of the area of the submanifold. 
In that light, as the critical points of the area functional, Riemannian submanifolds for which $H\equiv 0$ are particularly distinguished.
\begin{definition}
The submanifold $M$ as well as a corresponding immersion $\psi\colon \Sigma\to \ov{M}$ are called \textit{minimal} if
$H\equiv 0$. 
\end{definition}
\noindent Given local coordinates on $\Sigma$, the property that $H\equiv 0$ is equivalent to an elliptic, (in general) non-linear system of partial differential equations. To concretize the latter, we keep the assumption that $M$ is presented as the image $M_\psi=\psi(\Sigma)$ of a smooth immersion $\psi\colon \Sigma\to \ov{M}$ of an $m$-dimensional manifold $\Sigma$, which is equipped with the induced metric $g=\psi^*\langle\cdot,\cdot\rangle$. 
%
%
\begin{example}\label{meancurvatureRn}
Consider the case of $\ov{M}=\R^n$ equipped with the Euclidean metric. Then, the mean curvature vector of $\psi\colon \Sigma\to \R^n$ is given by
\begin{align*}
H^{\R^n}=\Delta_g\psi\,.
\end{align*}
Consequently, $\psi\colon\Sigma\to\R^n$ is a minimal immersion if and only if 
\begin{align*}
\Delta_g\psi=0\,.
\end{align*}
\end{example}
\begin{proof}
Let $(X_1,...,X_m)$ be a local, orthonormal frame on $M_\psi$. Then, we have
\begin{align*}
H^{\R^n}&=\sum_{i=1}^m \Big(\nabla_{X_i}^{\R^n}X_i\Big)^N
=\sum_{i=1}^m \nabla_{X_i}^{\R^n}X_i-\nabla_{X_i} X_i
=\sum_{i=1}^m X_i X_i-\nabla_{X_i}X_i\\[10pt]
&\hspace{1cm}=\sum_{i=1}^m X_i X_i \psi-\big(\nabla_{X_i}X_i\big)\psi
=\textup{tr}_g (\nabla\nabla \psi)
=\Delta_g\psi\,.\qedhere
\end{align*}
\end{proof}
\noindent The previous consideration immediately allows to generalize to ambient spaces given by embedded, Riemannian submanifolds of Euclidean space.
\begin{proposition}\label{meancurvaturesubmanifold}
Let $\ov{M}\subseteq\R^n$ be an embedded, Riemannian submanifold $\ov{M}$ of Euclidean space. Analogously as before, this induces an orthogonal splitting of the tangent spaces to $\R^n$ along $\ov{M}$. Denoting by $(\cdot)^{\ov{T}}$ the respective projection onto $T\ov{M}$, we have 
\begin{align*}
H^{\ov{M}}&=(\Delta_g\psi)^{\ov{T}}\,.
\end{align*}
\end{proposition}
\begin{proof}
This follows directly from the fact that the considered orthogonal projections pointwise commute. More precisely, using a local, orthonormal frame $(X_1,...,X_m)$ on $M_\psi$, we have
\begin{align*}
H^{\ov{M}}&=\sum_{i=1}^m \big(\ov{\nabla}_{X_i} X_i\big)^{N}=\sum_{i=1}^m \Bigg(\Big(\nabla^{\R^n}_{X_i} X_i\Big)^{\ov{T}}\Bigg)^{N}\\[10pt]
&\hspace{2cm}=\sum_{i=1}^m \Bigg(\Big(\nabla^{\R^n}_{X_i} X_i\Big)^{N}\Bigg)^{\ov{T}}=\Bigg(\sum_{i=1}^m \Big(\nabla^{\R^n}_{X_i} X_i\Big)^{N}\Bigg)^{\ov{T}}=(\Delta_g\psi)^{\ov{T}}\,.\qedhere
\end{align*}
\end{proof}
\noindent Thereby, we can deduce the minimal surface equation in the Euclidean $n$-sphere, that is, the embedded Riemannian submanifold
\begin{align*}
\mathbb{S}^n:=\Big\{P\in\R^{n+1}: |P|_{\R^{n+1}}=1\Big\}\subseteq\R^{n+1}.
\end{align*}
\begin{theorem}[\cite{Takahashi}]\label{minimalsurfaceequation}
$\psi\colon\Sigma\to \mathbb{S}^n\subseteq \R^{n+1}$ is a minimal immersion into $\mathbb{S}^n$ if and only if
\begin{align*}
\Delta_g \psi=-2\psi\,,
\end{align*}
that is, the coordinate functions of $\psi$ are eigenfunctions of the Laplace-Beltrami operater $\Delta_g$ with eigenvalue 2.
\end{theorem}
\begin{proof}
Let $(X_1,...,X_m)$ be a local, orthonormal frame on $M_\psi$. Then, for all the respective $P\in M_\psi$, we can identify
\begin{align*}
N_P\mathbb{S}^n&\cong\textup{span}(P)
\end{align*}
by regarding $\psi$ as the local vector field $\sum_{k=1}^m\psi^k X_k$. Consequently, by Theorem \ref{meancurvatureRn} and Proposition \ref{meancurvaturesubmanifold}, the immersion $\psi\colon\Sigma\to\mathbb{S}^n\subseteq\R^{n+1}$ is minimal in $\mathbb{S}^n$ if and only if 
there exists some $F\in C^\infty(\Sigma)$ such that
\begin{align*}
\Delta_g \psi=F\psi.
\end{align*}
Furthermore, using that $|\psi|^2\equiv 1$, we have
\begin{align*}
0&=\frac{1}{2} \Delta_g|\psi|^2=\langle\psi, \Delta_g\psi\rangle+ |\nabla\psi|^2=F|\psi|^2+ |\nabla\psi|^2=F+ |\nabla\psi|^2\,.
\end{align*}
Hence,
\begin{align*}
F=-|\nabla\psi|^2=-\sum_{i,k=1}^m X_i\big(\psi^k\big)^2=-\sum_{i,k=1}^m \big(\dup\psi(X_i)^k\big)^2=-\sum_{i,k=1}^m \big(X_i^k\big)^2=-\sum_{i=1}^m |X_i|^2=-m\,,
\end{align*}
and the theorem follows.
\end{proof}
\subsection{The Wedge Product on Euclidean $n$-Space}\label{wedge_product}
In this section we introduce the wedge product on $\R^n$ from a practical perspective (as opposed to the usual more abstract construction) and collect a few properties used in the following sections.

\bigskip
Let $m,\,n\in\Z_{\ge 1}$  with $m\le n$ and set $N:=\binom{n}{m}$. Consider the Euclidean space $\R^N$ and denote its standard basis by 
$\big(e_{i_1 ...i_m}\big)_{\substack{1\le i_1,...,i_m\le n,\\i_1<...<i_m}}$. 
\begin{definition}\label{wedgecoordinates}
The \textit{wedge product} of $m$ vectors in $\R^n$ is the $m$-linear, alternating map
\begin{align*}
\underbrace{\R^n\times ...\times \R^n}_{m\textup{ times}}\to \R^N,\quad (v_1,...,v_m)\mapsto v_1\wedge ...\wedge v_m\,,
\end{align*}
where the $i_1...i_m$-th component is defined as
\begin{align*}
\big(v_1\wedge ...\wedge v_m\big)^{i_1...i_m}:=\det\Bigg(\Big(v_l^{i_k}\Big)_{1\le k,l\le m}\Bigg)\,.
\end{align*}
\end{definition}
\noindent The definition of the wedge product directly implies that $v_1\wedge ...\wedge v_m=0$ whenever the vectors $v_1,..., v_m$ are linearly dependent.
Moreover, given any $v_1,...,v_m,\, w_1,...,w_m\in \R^n$, a computation shows that
\begin{align}
\langle v_1\wedge ...\wedge v_m, w_1\wedge ...\wedge w_m\rangle = \det\Big( \big(\langle v_i, w_j\rangle\big)_{1\le i,j\le n}\Big)\,.\label{wedgescalarproduct}
\end{align}
Particularly, we have $|v_1\wedge ...\wedge v_m|=1$ if $|v_1|=...=|v_m|=1$. 
Vectors of the form $v_1\wedge ...\wedge v_m$ are often referred to as $m$-\textit{vectors}. Using the wedge product, we have
\begin{align*}
e_{i_1 ...i_m}=e_{i_1}\wedge ...\wedge e_{i_m}\,.
\end{align*}
Hence, the description of Euclidean space $\R^N$ can be seen to arise from the abstract construction by linear combinations of the $m$-vectors $e_{i_1}\wedge ...\wedge e_{i_m}$ for which the scalar product is defined by the linear extension of \eqref{wedgescalarproduct}. In a broader context, this point of view exactly corresponds to the identification of $\R^N$ with the $m$-th component $\Lambda^m\R^n$ of the exterior algebra of $\R^n$.
\begin{definition}
Let $v_1\wedge ...\wedge v_m\in\Lambda^m\R^n$, $m\ge 1$. Then, its \textit{Hodge dual} $*(v_1\wedge ...\wedge v_m)\in \Lambda^{n-m}\R^n$ is defined to be the unique vector such that
\begin{align*}
\langle w_1\wedge ...\wedge w_{n-m},*(v_1\wedge ...\wedge v_m)\rangle=\det(v_1,...,v_m,w_1,...,w_{n-m})
\quad\textup{for all }w_1,...,w_m\in \R^n\,.
\end{align*}
\end{definition}
\noindent The Hodge dual $*(v_1\wedge ...\wedge v_{n-1})\in \Lambda^1\R^n=\R^n$ of $n-1$ linearly independent vectors $v_1,...,\,v_{n-1}\in\R^n$  has the following properties:
\begin{enumerate}[label=(\roman*)]
\item $*(v_1\wedge ...\wedge v_{n-1})=\sum_{i=1}^{n} \det(v_1,...,v_{n-1}, E_{i}) E_i$ for every orthonormal basis $(E_1,...,E_n)$ of $\R^n$;
\item $*(v_1\wedge ...\wedge v_{n-1})\in \textup{span}(v_1,...,v_n)^\perp$;
\item the family $\big(v_1,...,v_{n-1},*(v_1\wedge ...\wedge v_{n-1})\big)$ is positively oriented;
\item if $|v_1|=...=|v_{n-1}|=1$, then $|*(v_1\wedge ...\wedge v_{n-1})|=1$.
\end{enumerate}
Given $v_1,....,v_n\in \R^n$, then $*(v_1\wedge...\wedge v_n)\in\Lambda^0\R^n=\R$ is given by
\begin{align*}
*(v_1\wedge...\wedge v_n)=\det(v_1..., v_n)\,.
\end{align*}


\subsection{Bipolar Surface of a Minimal Surface in the 3-Sphere}
This section is included for the sake of self-containment and can be skipped by readers being familiar with the corresponding parts in \cite{Lawson}.

\bigskip
The aim of this section is to recap the derivations in \citep{Lawson} for immersed surfaces (i.e.\ two-dimensional, immersed submanifolds) in $\s^n$, more concretely in $\s^3$ and $\s^5$. In the following, let $M_\psi\subseteq\mathbb{S}^n$ be a surface described by an immersion $\psi\colon \Sigma\to\mathbb{S}^n$. The two-dimensional setup always allows to use a conformal atlas for $\Sigma$ such that with respect to local coordinates $\big(x^1,x^2\big)$, or with respect to the associated local frame $(\partial_1, \partial_2)$, the metric induced by $\psi$ is of the form
\begin{align}
g_{ij}:=2\lambda \delta_{ij},\quad i,j\in\{1,2\}\,,\label{conformality}
\end{align}
with a smooth $\lambda:\Sigma\to (0,\infty)$. In this case, the Gauss curvature of $(\Sigma,g)$ is given by
\begin{align}
K=-\frac{1}{4\lambda}\Delta \log(\lambda)\,.\label{Gauss_curvature}
\end{align}
In terms of the induced local frame $(\partial_1\psi, \partial_2\psi)$ on $M_\psi$  (i.e. $\partial_i\psi\colon=d\psi(\partial_i)$, $i=1,\,2$), equation \eqref{conformality} reads as
\begin{align*}
||\partial_1\psi||=||\partial_2\psi||=\sqrt{2\lambda}\quad\textup{and}\quad\langle\partial_1\psi,\partial_2\psi\rangle=0\,.
\end{align*}
Together with Proposition \ref{minimalsurfaceequation}, this shows that $\psi$ is minimal in $\mathbb{S}^n\subseteq\R^{n+1}$ if and only if
\begin{align}
\Delta \psi=-4\lambda\psi\,.\label{minimal_surface_eq_S3}
\end{align}
Now, by possibly passing to the orientable double cover, we can assume that $\Sigma$ is oriented.
In other words, $M_\psi$ can always be regarded as a conformally immersed Riemann surface $\Sigma$. In this context, we view local coordinates $\big(x^1, x^2\big)$ as one local complex coordinate $z=x^1+ix^2$ and define the local vector fields
\begin{align*}
\partial:=\frac{1}{2}(\partial_1-i\partial_2),\quad \ov{\partial}:=\frac{1}{2}(\partial_1+i\partial_2)\,. 
\end{align*}
and, correspondingly, the local vector fields $\partial\psi$ and $\ov{\partial}\psi$ on $M_\psi$. In these terms, the condition of conformality is expressed by
\begin{align}
\langle \partial\psi,\partial\psi\rangle=\langle\ov{\partial}\psi,\ov{\partial}\psi\rangle=0\quad\textup{and}\quad\langle\partial\psi,\ov{\partial}\psi\rangle=\lambda\,.\label{conformality_complex}
\end{align}
This completes the setup and notation we will use throughout the following.

\bigskip
At this point, we consider the codimension-1 case, i.e., immersed surfaces in $\s^3\subseteq\R^4$.
\begin{definition}\label{unit_normal}
Recalling the Hodge dual from Section \ref{wedge_product}, a unit normal vector field along $M_\psi$ that is tangent to $\mathbb{S}^3\subseteq\R^4$, can be defined by
\begin{align*}
\nu:=*\Bigg(\psi\wedge \frac{1}{\sqrt{2\lambda}}\partial_1\psi \wedge \frac{1}{\sqrt{2\lambda}}\partial_2\psi\Bigg)
=*\Bigg(\frac{1}{i\lambda}\psi\wedge\partial\psi\wedge\ov{\partial}\psi\Bigg)\,.
\end{align*}
\end{definition}
\begin{remark}
By using the components from Definition \ref{unit_normal} with respect to an orthonormal frame on $\R^4$, we consider the unit normal as a map $\nu\colon\Sigma\to \s^3\subseteq\R^4$, called the \textit{Gauss map} of $\psi$.
\end{remark}
\noindent The second fundamental form $B$ of $M_\psi$ is determined by the functions
\begin{align*}
\beta_{ij}:=\langle B_{ij}, \nu\rangle,\quad i,j\in\{1,2\}\,,
\end{align*}
where the $B_{ij}$ denote the components of $B$ with respect to local, conformal coordinates $\big(x^1,x^2)$. Since the latter are given by
\begin{align*}
B_{ij}^{\s^3}&=\Big(B_{ij}^{\R^4}\Big)^{||}\\
&=B_{ij}^{\R^4}-\langle B_{ij}^{\R^4}, \psi\rangle\psi\\
&=\Bigg[\partial_i\partial_j\psi - \frac{1}{2\lambda}\sum_{k=1,2} \langle\partial_i\partial_j \psi, \partial_k\psi\rangle\partial_k\psi\Bigg]-
\Bigg\langle \Bigg[\partial_i\partial_j\psi - \frac{1}{2\lambda}\sum_{k=1,2} \langle\partial_i\partial_j \psi, \partial_k\psi\rangle\partial_k\psi\Bigg],\psi\Bigg\rangle \psi\\
&=\partial_i\partial_j\psi - \frac{1}{2\lambda}\sum_{k=1,2} \langle\partial_i\partial_j \psi, \partial_k\psi\rangle\partial_k\psi-
\langle \partial_i\partial_j\psi,\psi\rangle\psi\\
&=\partial_i\partial_j\psi - \frac{1}{2\lambda}\sum_{k=1,2} \langle\partial_i\partial_j \psi, \partial_k\psi\rangle\partial_k\psi-
\lambda\delta_{ij}\psi,\quad i,j\in\{1,2\}\,, 
\end{align*}
the considerations from Section \ref{wedge_product} yield that
\begin{align*}
\beta_{ij}&=\det\Bigg(\psi,\frac{1}{\sqrt{2\lambda}}\partial_1\psi,\frac{1}{\sqrt{2\lambda}}\partial_2\psi, B_{ij}\Bigg)\\
&=\det\Bigg(\psi,\frac{1}{\sqrt{2\lambda}}\partial_1\psi,\frac{1}{\sqrt{2\lambda}}\partial_2\psi, \partial_i\partial_j\psi\Bigg)\\
&=*\Bigg(\psi \wedge \frac{1}{\sqrt{2\lambda}}\partial_1\psi \wedge \frac{1}{\sqrt{2\lambda}}\partial_2\psi\wedge \partial_i\partial_j\psi\Bigg)\,,
\end{align*}
or, in short
\begin{align*}
\beta_{ij}=\langle \nu,\partial_i\partial_j\psi\rangle\,.
\end{align*}
In this setting, the Gauss equation (cf.\ Proposition \ref{Gauss_equation}) writes as
\begin{align}
4\lambda^2(1-K)=\beta_{12}^2-\beta_{11}\beta_{22}.\label{Gauss_equation_codim1}
\end{align}
Moreover, the Weingarten equation (cf.\ Proposition \ref{Weingarten_equation}) takes the form
\begin{align}
\partial_i\nu =-\frac{1}{2\lambda}\sum_{k=1,2} \beta_{ik} \partial_k\psi,\quad i=1,2\,. \label{Weingarten_S3}
\end{align}
Now, the following lemma allows to characterize minimal surfaces in $\s^3$ by the means of a associated holomorphic function.
\begin{lemma}[\cite{Lawson}, Lemmata 1.2 and 1.4]\label{lemma_varphi}
If $\psi\colon\Sigma\to\s^3\subseteq\R^4$ is minimal, then the function
\begin{align}
\varphi:=\frac{1}{2}(\beta_{11}-i\beta_{12})=*\Bigg(\frac{1}{i\lambda} \psi\wedge \partial\psi\wedge\ov{\partial}\psi\wedge\partial^2\psi\Bigg)\label{varphi}
\end{align} 
is holomorphic on $\Sigma$. Moreover, we have
\begin{align}
|\varphi|^2=\lambda^2(1-K)\,.\label{varphi_abs}
\end{align}
Hence, if $\psi$ is minimal, then the Gauss curvature $K$ satisfies
\begin{align*}
K\le 1\quad\textup{and}\quad K=1\textup{ exactly at the isolated zeros of } \varphi\,.
\end{align*}
\end{lemma}
\begin{proof}
At first, we have
\begin{align*}
\beta_{22}=-\beta_{11}
\end{align*}
as $\psi$ is supposed to be minimal. Together with the Gauss equation \eqref{Gauss_equation_codim1}, this shows
\begin{align*}
|\varphi|^2=\frac{1}{4}\big(\beta_{11}^2+\beta_{12}^2\big)=\frac{1}{4}\big(\beta_{12}^2-\beta_{11}\beta_{22}\big)=\lambda^2(1-K)\,.
\end{align*}
To continue, since
\begin{align*}
\partial^2\psi=\frac{1}{4}\left(\partial_1^2\psi-\partial_2^2 \psi\right)-\frac{i}{2}\partial_1\partial_2\psi
\end{align*}
and
\begin{align*}
\partial_1^2\psi=4\lambda\psi-\partial_2^2\psi\,,
\end{align*}
by the minimal surface equation \eqref{minimal_surface_eq_S3}, we have
\begin{align*}
\frac{1}{2}\big(\partial_1^2\psi-i\partial_1\partial_2\psi\big)&= \frac{1}{4}\big(\partial_1^2\psi-\partial_2^2\psi-2i\partial_1\partial_2\psi\big)+\lambda \psi
=\partial^2\psi+\lambda\psi\,,
\end{align*}
implying that
\begin{align*}
\varphi&=\frac{1}{2}(\beta_{11}-i\beta_{12})\\
&=*\Bigg(\frac{1}{2\lambda} \psi\wedge\partial_1\psi\wedge\partial_2\psi\wedge\frac{1}{2}\big(\partial_1^2\psi-i\partial_1\partial_2\psi\big)\Bigg)\\
&=*\Bigg(\frac{1}{i\lambda} \psi\wedge\partial\psi\wedge\ov{\partial}\psi\wedge\big(\partial^2\psi+\lambda\psi\big)\Bigg)\\
&=*\Bigg(\frac{1}{i\lambda} \psi\wedge\partial\psi\wedge\ov{\partial}\psi\wedge \partial^2\psi\Bigg)\,.
\end{align*}
It remains to show that $\varphi$ is holomorphic, i.e., locally holomorphic around each point.
First, if $\varphi(p)=0$ for some $p\in\Sigma$, then we have that $K(p)=1$ by \eqref{varphi_abs}, that is, $\psi(p)$ is an umbilic point. Therefore, the second fundamental form and hence $\varphi$ itself vanish on a neighbourhood around $p$. So, particularly, $\varphi$ is holomorphic around $p$ in that case. Else, we use that
\begin{align}
\begin{aligned}
\langle\psi,\psi\rangle&=1\,,\\
\langle\partial\psi,\ov{\partial}\psi\rangle&=\lambda\,,\\
\langle \partial^i \psi,\partial^j \psi\rangle&=\langle \ov{\partial}^i \psi,\ov{\partial}^j \psi\rangle=0\quad\textup{for all } i,\, j\textup{ with } 1\le i+j\le 3\,.
\end{aligned}\label{scalar_products}
\end{align}
and compute
\begin{align*}
\varphi^2&=-\frac{1}{\lambda^2}\det\big(\psi,\partial\psi,\ov{\partial}\psi, \partial^2\psi\big)^2\\
&=-\frac{1}{\lambda^2}\det\left(
\begin{pmatrix}
\langle\psi,\psi\rangle &\langle\psi,\partial\psi\rangle &\langle\psi,\ov{\partial}\psi\rangle &\langle\psi,\partial^2\psi\rangle\\
\langle\partial\psi,\psi\rangle &\langle\partial\psi,\partial\psi\rangle &\langle\partial\psi,\ov{\partial}\psi\rangle &\langle\partial\psi,\partial^2\psi\rangle\\
\langle\ov{\partial}\psi,\psi\rangle &\langle\ov{\partial}\psi,\partial\psi\rangle &\langle\ov{\partial}\psi,\ov{\partial}\psi\rangle &\langle\ov{\partial}\psi,\partial^2\psi\rangle\\
\langle\partial^2\psi,\psi\rangle &\langle\partial^2\psi,\partial\psi\rangle &\langle\partial^2\psi,\ov{\partial}\psi\rangle &\langle\partial^2\psi,\partial^2\psi\rangle
\end{pmatrix}\right)\\
&=\langle\partial^2\psi,\partial^2\psi\rangle\,.
\end{align*}
Therewith, we find
\begin{align*}
2\varphi\cdot\ov{\partial}\varphi
&=\ov{\partial}\varphi^2\\
&=\ov{\partial}\langle\partial^2\psi,\partial^2\psi\rangle\\
&=2\langle \partial\big(\partial\ov{\partial}\psi\big),\partial^2\psi\rangle\\
&\overset{\eqref{minimal_surface_eq_S3}}{=}-2\langle \partial(\lambda\psi),\partial^2\psi\rangle\\
&=-2(\partial\lambda)\langle\psi,\partial^2\psi\rangle-2\lambda\langle\psi,\partial^2\psi\rangle\\
&=0\,,
\end{align*}
implying that $\ov{\partial}\varphi=0$ if $\varphi$ does not vanish on the considered local chart. Hence, also in this case, we have that $\varphi$ is locally holomorphic, which completes the proof.
\end{proof}
\begin{remark}\label{quadratic_differential}
As we will see later on (cf. \eqref{pnu_2}), we have
\begin{align*}
\varphi=-\langle\partial\psi,\partial\nu\rangle
\end{align*}
Hence, under a conformal change of the complex coordinate $z=x_1+ix_2$ to $w=y_1+iy_2$, we have
\begin{align*}
\varphi(z)=\varphi(w)\Bigg(\frac{\dup w}{\dup z}\Bigg)^2\,.
\end{align*}
Together with the property that $\varphi$ is holomorphic, $\varphi$ therefore defines a holomorphic quadratic differential $\omega:=\varphi \dup z^2$ on $\Sigma$, the so-called Hopf differential. As shown in \cite{Lawson}, this allows a further characterization of (closed) minimal surfaces in $\s^3$, particularly that the real projective plane $\R P^2$ cannot be minimally immersed into $\s^3$.
\end{remark}
\begin{lemma}[\cite{Lawson}, Remark 1.3]
If $\psi\colon\Sigma\to\s^3\subseteq\R^4$ is minimal, then the vector field 
$$\Phi:=(B_{11}-iB_{22})=\varphi\nu$$
satisfies
\begin{align}
\Phi=\lambda\partial\Bigg(\frac{1}{\lambda}\partial\psi\Bigg)\label{Phi1}
\end{align}
and
\begin{align}
\ov{\partial}\Phi=-(1-K)\lambda\partial\psi\,.\label{Phi2}
\end{align}
\end{lemma}
\begin{proof}
Denote by $(E_1,E_2,E_3,E_4)$ an orthonormal frame on $\R^4$. To show that the first equation holds, we use \eqref{varphi} as well as \eqref{scalar_products} and compute
\begin{align*}
\Phi&=\varphi\nu\\
&=*\Bigg(\frac{1}{i\lambda}\psi\wedge\partial\psi\wedge\ov{\partial}\psi\wedge\partial^2\psi\Bigg)*\Bigg(\frac{1}{i\lambda}\psi\wedge\partial\psi\wedge\ov{\partial}\psi\Bigg)\\
&=\det\Bigg(\frac{1}{i\lambda}\psi,\partial\psi,\ov{\partial}\psi,\partial^2\psi\Bigg)\sum_{i=1}^4 \det\Bigg(\frac{1}{i\lambda}\psi,\partial\psi,\ov{\partial}\psi,E_i\Bigg)E_i\\
&=-\frac{1}{\lambda^2}\sum_{i=1}^4 
\det\left(
\begin{pmatrix}
1 & 0 & 0 & \psi^i\\
0 & 0 &\lambda &\partial\psi^i\\
0 &\lambda & 0 &\ov{\partial}\psi^i\\
0 & 0 &\partial\lambda &\partial^2\psi^i
\end{pmatrix}
\right)
E_i\\
&=-\frac{1}{\lambda}(\partial\lambda)\partial\psi+\partial^2\psi\\
&=\lambda\partial\Bigg(\frac{1}{\lambda}\partial\psi\Bigg)\,.
\end{align*}
For the second equation, we use that $\varphi$ is holomorphic, i.e, $\ov{\partial}\varphi=0$, and obtain
\begin{align*}
\ov{\partial}\Phi&=\ov{\partial}(\varphi\nu)\\
&=\varphi\ov{\partial}\nu\\
&=*\Bigg(\frac{1}{i\lambda}\psi\wedge\partial\psi\wedge\ov{\partial}\psi\wedge\partial^2\psi\Bigg)*\Bigg(\frac{1}{i\lambda}\psi\wedge\partial\psi\wedge\ov{\partial}^2\psi\Bigg)\\
&=-\frac{1}{\lambda^2}\sum_{i=1}^4 
\det\left(
\begin{pmatrix}
1 & 0 & 0 & \psi^i\\
0 & 0 &\lambda &\partial\psi^i\\
0 &\lambda & 0 &\ov{\partial}\psi^i\\
0 & 0 &\lambda^2+\partial\ov{\partial}\lambda &\ov{\partial}^2\psi^i
\end{pmatrix}
\right)
E_i\\
&=\Bigg(-\lambda-\frac{\partial\ov{\partial}\lambda}{\lambda}\Bigg)\partial\psi\\
&\overset{\eqref{Gauss_curvature}}{=}-(1-K)\lambda\partial\psi\,.
\end{align*}
\end{proof}
\begin{lemma}[\cite{Lawson}, p.360]\label{polar_surface_conformal}
If $\psi\colon\Sigma\to\s^3\subseteq\R^4$ is minimal, then its Gauss map $\nu\colon\Sigma\to \s^3\subseteq\R^4$ is almost conformal with induced metric 
\begin{align*}
g^{\nu}=(1-K)g\,.
\end{align*}
%
Particularly, $\nu$ has singularities precisely at the isolated points where $K=1$. 
\end{lemma}
\begin{proof}
At first, we have
\begin{align}
\partial\nu&=\partial*\Bigg(\frac{1}{i\lambda}\psi\wedge\partial\psi\wedge\ov{\partial}\psi\Bigg)\notag\\
&=*\Bigg[\partial\Bigg(\frac{1}{i\lambda}\psi\Bigg)\wedge\partial\psi\wedge\ov{\partial}\psi
+\frac{1}{i\lambda} \psi\wedge\partial^2\psi\wedge\ov{\partial}\psi+\frac{1}{i\lambda}\psi\wedge\partial\psi\wedge \partial\ov{\partial}\psi\Bigg]\notag\\
&\overset{\eqref{minimal_surface_eq_S3}}{=} *\Bigg[\psi\wedge\Bigg(\Bigg(\partial\frac{1}{i\lambda}\Bigg)\partial\psi+\frac{1}{i\lambda}\partial^2\psi\Bigg)\wedge\ov{\partial}\psi\Bigg]\notag\\
&=*\Bigg[\psi\wedge\partial\Bigg(\frac{1}{i\lambda}\partial\psi\Bigg)\wedge\ov{\partial}\psi\Bigg]\notag\\
&\overset{\eqref{Phi1}}{=}*\Bigg(\frac{1}{i\lambda}\psi\wedge\Phi\wedge\ov{\partial}\psi\Bigg)\,.\label{pnu_1}
\end{align}
Now, by definition we have $\Phi=\varphi\nu$. Hence, we continue by
\begin{align}
\partial\nu&=\varphi*\Bigg(\frac{1}{i\lambda}\psi\wedge\nu\wedge\ov{\partial}\psi\Bigg)\notag\\
&=\varphi \cdot\Bigg[*\Bigg(\frac{1}{2\lambda}\psi\wedge\nu\wedge\partial_2\psi\Bigg)-i*\Bigg(\frac{1}{2\lambda}\psi\wedge\nu\wedge\partial_1\psi\Bigg)\Bigg]\notag\\
&=\varphi \cdot\Bigg[*\Bigg(\frac{1}{\sqrt{2\lambda}}\psi\wedge\nu\wedge\frac{1}{\sqrt{2\lambda}}\partial_2\psi\Bigg)-i*\Bigg(\frac{1}{\sqrt{2\lambda}}\psi\wedge\nu\wedge\frac{1}{\sqrt{2\lambda}}\partial_1\psi\Bigg)\Bigg]\notag\\
&=\varphi \frac{1}{\sqrt{2\lambda}}\Bigg(-\frac{1}{\sqrt{2\lambda}}\partial_1\psi-i\frac{1}{\sqrt{2\lambda}}\partial_2\psi\Bigg)\notag\\
&=-\frac{1}{\lambda}\varphi\ov{\partial}\psi\,.\label{pnu_2}
\end{align}
Together with \eqref{conformality_complex}, this implies
\begin{align*}
\langle \partial\nu,\partial\nu\rangle&= \frac{\varphi^2}{\lambda^2}\langle\partial\psi,\partial\psi\rangle=0,\\
\langle \ov{\partial}\nu,\ov{\partial}\nu\rangle&= \frac{\ov{\varphi}^2}{\lambda^2}\langle\ov{\partial}\psi,\ov{\partial}\psi\rangle=0\\
\end{align*}
as well as
\begin{align}
\langle \partial\nu,\ov{\partial}\nu\rangle=\frac{|\varphi|^2}{\lambda^2}\langle\partial\psi,\ov{\partial}\psi\rangle=\frac{|\varphi|^2}{\lambda^2}\lambda
\overset{\eqref{varphi_abs}}{=} (1-K)\lambda\,.\label{conf_nu}
\end{align}
Therefore, $\nu$ is almost conformal with $\lambda^\nu:=(1-K)\lambda$, i.e., 
\begin{align*}
g^\nu=2\lambda^\nu\delta =(1-K)g\,, 
\end{align*}
and clearly, $\lambda^\nu$ vanishes exactly at the isolated points with $K=1$. 
%
\end{proof}
\begin{theorem}[\cite{Lawson}, Proposition 10.1]\label{polar_variety}
If $\psi\colon\Sigma\to\s^3\subseteq\R^4$ is minimal, then its Gauss map $\nu\colon\Sigma\to \s^3\subseteq\R^4$ is again minimal in the sense that the minimal surface equation is satisfied. In this case, the (in general branched) surface described by $\nu$ is called the \textit{polar variety} of $M_\psi$.
\end{theorem}
\begin{proof}{(of Theorem \ref{polar_variety})}
Using \eqref{pnu_1}, the complex conjugate of \eqref{Phi1} as well as \eqref{Phi2}, we have
\begin{align*}
\partial\ov{\partial}\nu&= \ov{\partial}\Bigg[ *\Bigg(\frac{1}{i\lambda}\psi\wedge\Phi\wedge\ov{\partial}\psi\Bigg)\Bigg]\\
&=*\Bigg[\psi\wedge\ov{\partial}\Phi\wedge\frac{1}{i\lambda}\ov{\partial}\psi+\psi\wedge\Phi\wedge\ov{\partial}\Bigg(\frac{1}{i\lambda}\ov{\partial}\psi\Bigg)\Bigg]\\
&=-(1-K)\lambda*\Bigg(\frac{1}{i\lambda}\psi\wedge\partial\psi\wedge\ov{\partial}\psi\Bigg)+*\Bigg(\frac{1}{i\lambda}\psi\wedge\Phi\wedge\ov{\Phi}\Bigg)\\
&=-(1-K)\lambda\nu+*\Bigg(\frac{1}{i\lambda}\psi\wedge\Phi\wedge\ov{\Phi}\Bigg)\\
&= -(1-K)\lambda\nu
\end{align*}
where we used in the last step that
\begin{align*}
*\Bigg(\frac{1}{i\lambda}\psi\wedge\Phi\wedge\ov{\Phi}\Bigg)&=*\Bigg(\frac{|\varphi|^2}{i\lambda}\psi\wedge\nu\wedge\nu\Bigg)=0\,.\\
\end{align*}
Concluded, we have
\begin{align*}
\partial\ov{\partial}\nu=-\lambda^\nu\nu\,.
\end{align*}
\end{proof}
\noindent As shown \cite{Lawson}, we include the following property of the polar variety which will be used later on. The corresponding proof is based on the analysis of the quadratic differential $\omega$ from Remark \ref{quadratic_differential}.
\begin{proposition}[\cite{Lawson}, p. 361]\label{branchpoints}
The Gauss map $\nu\colon\Sigma\to\s^3\subseteq\R^4$ of a minimal immersion is non-singular if and only if $\Sigma$ covers a torus or a Klein bottle.
\end{proposition}
\noindent At this point, we continue with another map associated to every immersion into $\s^3$. 
\begin{definition}[\cite{Lawson}, p. 361]\label{bipolar_surface}
In the notation of Section \ref{wedge_product}, we define the map
\begin{align*}
\widetilde{\psi}\colon\Sigma\to\mathbb{S}^5\subseteq\R^6\cong\Lambda^2\R^4,\quad\widetilde{\psi}:=\psi\wedge\nu\,.
\end{align*}
\end{definition}
\begin{lemma}[\cite{Lawson}, p.361]\label{bipolar_surface_conformal}
If $\psi\colon\Sigma\to\s^3\subseteq\R^4$ is minimal, then $\widetilde{\psi}\colon\Sigma\to \s^5\subseteq\R^6$ is a non-singular, conformal immersion with induced metric 
\begin{align*}
g^{\widetilde{\psi}}=(2-K)g\,.
\end{align*}
Hence, if $\Sigma$ is closed, then the Gauss-Bonnet theorem immediately yields
\begin{align*}
\textup{area}\Big(\widetilde{\psi}\Big)=2\textup{area}(\psi)-2\pi\chi(\Sigma)\,.
\end{align*}
\end{lemma}
\begin{proof}
Using the product rule for bilinear maps, we compute
\begin{align*}
\partial\widetilde{\psi}=\partial\psi\wedge\nu+\psi\wedge\partial\nu\,,
\ov{\partial}\widetilde{\psi}=\ov{\partial}\psi\wedge\nu+\psi\wedge\ov{\partial}\nu\,.
\end{align*}
By the means of \eqref{scalar_products} and \eqref{pnu_2}, we obtain
\begin{align*}
\langle \partial\widetilde{\psi},\partial\widetilde{\psi}\rangle&=
\langle \partial\psi\wedge\nu, \partial\psi\wedge\nu\rangle +2\langle\partial\psi\wedge\nu,\psi\wedge\partial\nu\rangle+\langle \psi\wedge\partial\nu,\psi\wedge\partial\nu\rangle\\
&=\det\left(\begin{pmatrix}
\langle \partial\psi,\partial\psi\rangle & \langle \partial\psi,\nu\rangle\\
\langle\nu, \partial\psi\rangle & \langle \nu,\nu\rangle
\end{pmatrix} 
\right)
+ 2\cdot\det\left(\begin{pmatrix}
\langle \partial\psi, \psi\rangle & \langle \partial\psi,\partial\nu\rangle\\
\langle \nu,\psi\rangle & \langle\nu,\partial\nu\rangle
\end{pmatrix}
\right)\\
&\hspace{1.5cm}+\det\left(\begin{pmatrix}
\langle\psi,\psi\rangle & \langle \psi,\partial\nu\rangle\\
\langle\partial\nu,\psi\rangle & \langle\nu,\partial\nu\rangle
\end{pmatrix}
\right)\\
&=0
\end{align*}
and analogously, by \eqref{conf_nu},
\begin{align*}
\langle \ov{\partial}\widetilde{\psi},\ov{\partial}\widetilde{\psi}\rangle=0,\quad
\langle \partial\widetilde{\psi},\ov{\partial}\widetilde{\psi}\rangle=(2-K)\lambda\,.
\end{align*}
As $K\le 1$, $\widetilde{\psi}$ is conformal with $\lambda^{\widetilde{\psi}}:=(2-K)\lambda$, i.e., 
\begin{align*}
g^{\widetilde{\psi}}=2\lambda^{\widetilde{\psi}}\delta =(2-K)g\,.
\end{align*}
\end{proof}
\begin{theorem}[\cite{Lawson}, p. 361]\label{minimal_bipolar_surface}
If $\psi\colon\Sigma\to\s^3\subseteq\R^4$ is minimal, then the map $\widetilde{\psi}\colon\Sigma\to\mathbb{S}^5\subseteq\R^6$ is again minimal and describes the  so-called bipolar surface of $M_\psi$.
\end{theorem}
\begin{proof}
At first, the Weingarten equation yields that
\begin{align*}
|\partial_1\psi\wedge\partial_1\nu|^2 =|\partial_1\psi|^2|\partial_1\nu|^2-\langle\partial_1\psi,\partial_1\nu\rangle^2
=\beta_{12}^2
\end{align*}
and, similarly, 
\begin{align*}
|\partial_2\psi\wedge\partial_2\nu|^2=\beta_{12}^2
\end{align*}
as well as
\begin{align*}
\langle\partial_1\psi\wedge\partial_1\nu,\partial_2\psi\wedge\partial_2\nu\rangle=-\beta_{12}^2\,.
\end{align*}
From this, we deduce that
\begin{align*}
|\partial_1\psi\wedge\partial_1\nu+\partial_2\psi\wedge\partial_2\nu|^2=|\partial_1\psi\wedge\partial_1\nu|^2+|\partial_2\psi\wedge\partial_2\nu|^2+2\langle\partial_1\psi\wedge\partial_1\nu,\partial_2\psi\wedge\partial_2\nu\rangle=0\,.
\end{align*}
In particular, we have that
\begin{align*}
\ov{\partial}\psi\wedge\partial\nu+\partial\psi\wedge\ov{\partial}\nu=2\cdot(\partial_1\psi\wedge\partial_1\nu+\partial_2\psi\wedge\partial_2\nu)=0\,.
\end{align*}
Together with the product rule and the minimal surface equations for $\psi$ and $\nu$, this finally yields that
\begin{align*}
\partial\ov{\partial}\widetilde{\psi}&=\partial\ov{\partial}\psi\wedge\nu+\ov{\partial}\psi\wedge\partial\nu+\partial\psi\wedge\ov{\partial}\nu+\psi\wedge \partial\ov{\partial}\nu\\
&=-(2-K)\lambda\widetilde{\psi}\,,
\end{align*}
i.e., by Lemma \ref{bipolar_surface_conformal}, that the conformal immersion $\widetilde{\psi}$ is again minimal. 
\end{proof}
\begin{remark}
For closed minimal surfaces the topology of the surface $M_\psi$ is determined by an immersion $\psi\colon\Sigma\to\s^3$ on a closed, smallest and possibly non-orientable \textit{fundamental domain} $\Sigma$ such that $\psi$ is injective up to the occurrence of local self-intersections. The area measured by an immersion of this kind already gives the actual area of the surface $M_\psi$, which, otherwise, could be hidden by multiple coverings of $M_\psi$.
Now, defining the bipolar immersion $\widetilde{\psi}\colon\Sigma\to\s^5$ on a smallest fundamental domain $\Sigma$ for $M_\psi$ (or, if necessary, its orientable double-cover with the specified orientation-reversing involution) it is not clear in the first place whether $\Sigma$ is also a smallest fundamental domain for the bipolar surface $\widetilde{\psi}(\Sigma)$. In fact, as shown by the examples of Lawson's bipolar $\tau$-surfaces in \cite{Lapointe}, the latter is not necessarily the case: the bipolar surface can be multiply covered by $\Sigma$ and performing the quotient with respect to such covering maps on $\Sigma$ can lead to a change of the topology. In particular, this indicates that the formula given in Lemma \ref{bipolar_surface_conformal} must be read as an upper bound to the actual area of the bipolar surface. Concerning the results in \cite{Lapointe}, the information on fundamental domains for bipolar immersions was deduced by symmetries acting on the domain of the bipolar immersion, heavily relying on the knowledge of explicit parametrizations. 
\end{remark}

\section{Minimal Immersions arising from Lawson's \\Construction Procedure}\label{preimage_construction}
In the following, the goal is to analyse the bipolar surfaces of closed minimal surfaces in $\s^3$ resulting from Lawson's construction procedure introduced in \cite{Lawson}. By successive application of the Schwarz reflection principle, this method extends an embedded minimal disk in $\mathbb{S}^3\subseteq \R^4$, bounded by a geodesic polygon $\Gamma$, to a complete, non-singular minimal surface $\mathcal{M}_\Gamma\subseteq \s^3$  (cf.\ \citep{Lawson}, Theorem 1). The key ingredient of this existence result is the chosen class of geodesic polygons in $\mathbb{S}^3$, which e.g.\ requires the interior angles to be of the form $\frac{\pi}{k}$ for some $k\in \Z_{\ge 2}$ and that $\Gamma$ is \textit{convex}, i.e., is contained in the boundary of its convex hull
\begin{align*}
\mathcal{C}(\Gamma):=\bigcap\big\{H\subseteq\s^3: H\textup{ is a closed hemisphere containing }\Gamma\big\}
\end{align*}
(for a full characterization see \cite{Lawson}, p. 341). Considered as boundary values for the Plateau problem, such polygons guarantee the existence of a unique solution $f:\Delta\to \s^3$ defined on the closed unit disk $\Delta\subseteq\R^2$ (cf.\ \cite{Morrey}, \cite{MeeksYau}), which is a conformal embedding from the interior $\Delta^\circ$ into $\mathcal{C}(\Gamma)^\circ$ by \cite{Lawson_Sn} and one-to-one on $\partial\Delta$ (cf.\ \cite{Courant}). Due to \cite{HeinzHildebrandt}, $f$ is in addition analytic on $\partial\Delta$ except possibly at the points corresponding to the vertices of $\Gamma$. Therewith, $f$ can be analytically extended by reflections across the geodesic arcs of $\Gamma$ (cf. \cite{Lawson}, Proposition 3.1), and repeated application of this principle generates the reflection process.

We start this section by translating Lawson's construction of a closed minimal surface $\mathcal{M}_\Gamma$ into a definition of an associated immersion $\psi\colon S\to\s^3$ on a smallest fundamental domain $S$. Next, by a possible transition to the orientable double cover $\ov{S}$ of $S$, we will specify a corresponding Gauss map $\nu$, which then allows us to define the bipolar immersion $\widetilde{\psi}$. Finally, we will use this setup to deduce a general condition for the existence of a smooth, orientation-preserving covering map on $S$ or $\ov{S}$ which leaves $\widetilde{\psi}$ invariant. As shown later on, this can provide an approach to a smallest fundamental domain of the bipolar immersion or, in other words, to the topology of the bipolar surface.

\bigskip
To begin, suppose that $\Gamma\subseteq\s^3$ is a geodesic polygon meeting the requirements from \cite{Lawson} and that $f: \Delta\to \mathbb{S}^3$ is the unique parametrization of the initial piece of surface which was mentioned before. Denote by $\gamma_1, ..., \gamma_N$ the great circles containing the boundary arcs of $\Gamma$. Let $r_{\gamma_1}, ...,r_{\gamma_N}\in \textup{SO}(4)$ be the corresponding geodesic reflections and $G=\langle r_{\gamma_1}, ...,r_{\gamma_N}\rangle $ their freely generated group. As deduced in \cite{Lawson}, the successive reflection process yields
\begin{align*}
\mathcal{M}_\Gamma=\bigcup_{g\in G} (g\circ f)(\Delta)\,,
\end{align*}
which particularly shows that $\mathcal{M}_\Gamma$ is closed if and only if the group $G$ is finite. In that sense, we stick to the assumption that $|G|<\infty$ in what follows.

In order to obtain the desired domain $S$, the idea is now to glue together the preimages of the minimal disks $g\circ f$ for $g\in G$. More precisely, this means to introduce an equivalence relation on the stack $G\times \Delta$ of labelled disks $\{g\}\times \Delta$ in accordance with Lawson's reflection process. Clearly, if two points $(g,p)$ and $(h,q)\in G\times \Delta$ are identified in this way, we must have 
\begin{align*}
(g\circ f)(p)=(h\circ f)(q)\,.
\end{align*}
But to define the gluing relation, this condition is not sufficient since the equality could as well correspond to a self-intersection of the surface. To exclude the latter scenario, we put an additional condition on the group elements $g$ and $h$, which is derived as follows. Consider for each $p\in \Delta$ the subgroup 
\begin{align}
G^p:=\Big\langle \{r_{\gamma_i}: f(p)\in\gamma_i\}\Big\rangle\subseteq G\,. \label{G^p}
\end{align}
Regarding the group structure of $G^p$, only the following three cases can occur:
\begin{enumerate}
\item If $p\in \Delta^\circ$, then $G^p=\{\mathbbm{1}_4\}$.
\item If $p\in \partial\Delta$ is not mapped onto a vertex of $\Gamma$ by $f$, i.e. $f(p)\in \gamma_i$ for exactly one $i\in\{1,...,N\}$, then 
\begin{align*}
G^p=\{\mathbbm{1}_4, r_{\gamma_i}\}\cong \Z_2.
\end{align*}
\item If $p\in \partial\Delta$ is mapped onto a vertex of $\Gamma$ by $f$, i.e. $f(p)\in \gamma_i\cap \gamma_{i+1}$ for an $i\in\{1,...,N\}$, then
\begin{align*}
G^p=\langle r_{\gamma_i}, r_{\gamma_{i+1}}\rangle\cong D_n
\end{align*}
where $n$ defines the interior angle $\frac{\pi}{n}$  of $\Gamma$ at $f(p)$ and $D_n$ denotes the dihedral group of order $2n$. This is clear by the fact that $r_{\gamma_{i+1}}\cdot r_{\gamma_i}$ is the rotation by $\frac{2\pi}{n}$ around the great circle through $f(p)$ which is orthogonal to $r_{\gamma_{i+1}}$ and $r_{\gamma_i}$.
\end{enumerate}
For the initial piece $f$, the group $G^p$ exactly labels the different, neighbouring pieces of surface at the point $f(p)$, which, according to \cite{Lawson}, yield an analytic, non-singular extension of $f$ in a neighbourhood of $p$ (cf.\ \cite{Lawson}, Lemmata 4.2 and 4.3). As a consequence, $g\cdot G^p$ encodes the neighbours at the point $(g\circ f)(p)$ for any $g\in G$. Thus, we conclude that two points $(g,p),~(h,q)\in G\times \Delta$ that satisfy
\begin{align*}
(g\circ f)(p)=(h\circ f)(q)
\end{align*}
belong to neighbouring pieces if and only if $g^{-1}\cdot h\in G^p$. Concluded, this allows the following construction.
\begin{construction}{}\label{construction}
On $G\times\Delta$, we introduce the equivalence relation
\begin{align*}
(g,p)\sim (h,q)~:\Leftrightarrow~ (g\circ f)(p)=(h\circ f)(q)~\textup{and}~g^{-1}\cdot h\in G^p
\end{align*}
and denote a corresponding equivalence class by $\big[(g,p)\big]$. Then, due to \cite{Lawson} (Section 4), the quotient
\begin{align*}
S:=\big(G\times\Delta\big)/\sim
\end{align*}
is a closed, smooth 2-manifold and 
\begin{align*}
\psi\colon S\to\mathbb{S}^3,~\psi\big([(g,p)]\big):=(g\circ f)(p)
\end{align*}
is an immersion of the closed minimal surface $\mathcal{M}_\Gamma$. Moreover, if the subgroup $G^{\Gamma}\subseteq G$ of symmetries of $\Gamma$ is trivial, then $\psi$ immerses $\mathcal{M}_\Gamma$ on a smallest fundamental domain.
\end{construction}
\begin{remark}\label{coveringmapGamma}
In the previous setup, a non-trivial subgroup $G^{\Gamma}\subseteq G$ of symmetries of $\Gamma$ induces a smooth covering map of degree $\big|G^\Gamma\big|$ on $S$ under which the immersion $\psi$ is invariant. More precisely, the latter is given by the orbit space projection of the group action
\begin{align*}
G^\Gamma\times S\to S,~\big(h, [(g,p)]\big)\mapsto \Big[\Big(gh, \big(f^{-1}\circ h^{-1}\circ f\big)(p)\Big)\Big]\,,
\end{align*}
where $f$ is understood as a homeomorphism onto its image. This action is smooth, proper (since $G^\Gamma$ is finite) and free (since $G^\Gamma\cap G^p=\{\mathbbm{1}_4\}$ for all $p\in\Delta$), and hence meets the requirements from Theorem 7.13 of \cite{Lee}.

A more practicable way to handle a non-trivial subgroup $G^{\Gamma}\subseteq G$ of symmetries of $\Gamma$ is to pass to a smaller initial piece of surface, bounded by a smaller geodesic polygon $\Gamma'$, such that the corresponding group $G'$ generated by Schwarz reflections does no longer contain a non-trivial subgroup leaving $\Gamma'$ invariant. 
\end{remark}
\begin{corollary}[cf. \cite{Lawson}, Prop. 4.4]\label{Eulercharacteristic}
The Euler characteristic $\chi(S)$ of $S$ is fully determined by the polygon $\Gamma$: We have
\begin{align*}
\chi(S)=\frac{|G|}{\big|G^\Gamma\big|}\cdot \Bigg( 1-\sum_{i=1}^N  \frac{k_i-1}{2 k_i}\Bigg)\,,
\end{align*}
where $\frac{\pi}{k_1}.,..., \frac{\pi}{k_N}$, $k_i\in \Z_{\ge 2}$, denote the interior angles of $\Gamma$.
\end{corollary}
\begin{proof}
The local version of the Gauss-Bonnet theorem yields 
\begin{align*}
\int_\Delta K^f\, dA&=2\pi-\sum_{i=1}^N  \Bigg(\pi-\frac{\pi}{k_i}\Bigg)=2\pi \Bigg( 1-\sum_{i=1}^N  \frac{k_i-1}{2 k_i}\Bigg)\,.
\end{align*}
Since
\begin{align*}
\int_S K^\psi\, dA= \frac{|G|}{\big|G^\Gamma\big|}\cdot \int_\Delta K^f\,dA\,,
\end{align*}
the global version implies
\begin{align*}
2\pi \chi(S)=\int_S K^\psi\, dA
\end{align*}
and therefore the assertion follows.
\end{proof}

\noindent For the topological classification of $S$, it remains to check whether $S$ is orientable or not. To this end, we consider a Gauss map $n: \Delta\to\mathbb{S}^3$ associated to the embedding $f$. Under the application of a Schwarz reflection across an arc of $\Gamma$, performed by a rotation of 180$^\circ$ about the respective great circle, the initial normal $n$ can be extended by the same reflection as $f$ combined with an additional flip. Concerning the application of a product of the generators $r_{\gamma_1},...,r_{\gamma_N}$, this generalizes in the sense that each generator contributes one flip. Consequently, $S$ is non-orientable if and only if we find a sequence of the generators that starts and ends at the initial piece $f$ but returns with the opposite normal $-n$. Expressed in terms of the group $G$ and the repeated flips of $n$, the latter situation is described by the following proposition.
\begin{proposition}{}\label{nonorientability}
$S$ is non-orientable if and only if the identity $e_G\in G$ can be written as an odd number of the generators $r_{\gamma_1},...,r_{\gamma_N}$.
\end{proposition}
Therefore, if $S$ is orientable, any representation of a group element $g\in G$ in terms of the considered generators does either contain even \textit{or} odd numbers of the latter. So, the following quantity is well-defined.
\begin{definition}{}\label{parity}
If $S$ is orientable, then we define the \textit{parity} $\sigma(g)$ of $g\in G$ by
\begin{align*}
\sigma(g):=\begin{cases}
~0\\
~1
\end{cases}
\hspace{0.1cm}\textup{ if $g$ contains }\hspace{0.1cm}
\begin{matrix} 
\textup{ even }\\
\textup{ odd }
\end{matrix}\hspace{0.1cm}\textup{ numbers of the generators $r_{\gamma_1},...,r_{\gamma_N}$}.
\end{align*}
\end{definition}
\noindent Therewith, the extension of an intial unit normal $n$ writes as follows.
\begin{construction}\label{normal_oriented}
Let $n\colon\Delta\to \mathbb{S}^3$ be a Gauss map of $f$. 
If $S$ is orientable, then 
\begin{align*}
\nu\colon S\to \mathbb{S}^3,~\nu\big([(g,p)]\big):=(-1)^{\sigma(g)}\cdot(g\circ n)(p)
\end{align*}
is a Gauss map of the immersion $\psi\colon S\to\s^3$ from Construction \ref{construction}. The choice of $\nu$ induces an orientation form $\omega$ on $S$ by the pullback of an orientation form $\Omega$ on $\s^3\subseteq \R^4$. We set
\begin{align}
\omega|_{x}(v,w):=\Omega|_{\psi(x)}\Big(d\psi|_{x}(v),d\psi|_{x}(w),\nu(x)\Big):=\det\Big(\psi(x),d\psi|_{x}(v),d\psi|_{x}(w),\nu(x)\Big)\label{orientationS}
\end{align}
for $x\in S$ and $v,w\in T_xS$.
\end{construction}
\noindent To describe the bipolar surface of $\mathcal{M}_\Gamma$, we additionally need to make sense of a Gauss map in the non-orientable case. This is performed by the transition to the orientable double cover of $S$, which can be, in a straightforward manner, constructed in similar way as $S$.
\begin{construction}\label{doublecover}
Let $n\colon\Delta\to \mathbb{S}^3$ be a Gauß map of $f$. On $\Z_2\times G\times\Delta$, we define the equivalence relation
\begin{align*}
(s,g,p)\sim_{\textup{dc}}(t,h,q)~:\Leftrightarrow~ (g,p)\sim (h,q)~\textup{and}~ (-1)^s \cdot (g\circ n)(p)=(-1)^t\cdot (h\circ n)(q)
\end{align*}
and denote a corresponding equivalence class by $[(s,g,p)]$. If $S$ is non-orientable, the orientable double cover of $S$ is given by the smooth 2-manifold
\begin{align*}
\ov{S}:=\big( \Z_2\times G\times\Delta\big)/\sim_{\textup{dc}}\,,
\end{align*}
where the map
\begin{align*}
i\colon \ov{S}\to \ov{S},~i\big([(s,g,p)]\big):=[(-s,g,p)]
\end{align*}
is an orientation-reversing, fixed-point-free involution satisfying that $\ov{S}/\langle i\rangle=S$. In this case, we describe the surface $\mathcal{M}_\Gamma$ by the immersion
\begin{align*}
\ov{\psi}\colon\ov{S}\to\s^3,~\ov{\psi}\big([(s,g,p)]\big):=\psi\big([g,p]\big)
\end{align*}
with Gauss map
\begin{align*}
\ov{\nu}:\ov{S}\to\s^3,~\ov{\nu}\big([(s,g,p)]\big):=(-1)^{s}\cdot(g\circ n)(p)\,.
\end{align*}
The induced orientation form $\ov{\omega}$ on $\ov{S}$ is defined analogously as in \eqref{orientationS}. 

\bigskip
Concluded, the bipolar surface $\widetilde{\mathcal{M}}_\Gamma\subseteq \s^5$ of $\mathcal{M}_\Gamma$ is immersed by 
\begin{align*}
\widetilde{\psi}:=\begin{cases}
~\psi\wedge\nu\\
~\ov{\psi}\wedge \ov{\nu}
\end{cases}
\hspace{0.2cm}\textup{if $S$ is }\hspace{0.3cm}
\begin{matrix} 
\textup{orientable}\,;\\
\textup{non-orientable.}
\end{matrix}
\end{align*}
\end{construction}
\begin{remark}
If $S$ is orientable, then $\ov{S}$ is the disconnected union of two copies of $S$ (and not relevant for us in the following).
\end{remark}
\noindent Towards a smallest domain for the bipolar immersion, the definition of $S$ or $\ov{S}$ allows to exploit the bilinearity of the wedge product as follows.
\begin{theorem}\label{-1_cover}
~
\begin{enumerate}[label=(\roman*)]
\item If $S$ is orientable and $-\mathbbm{1}_4\in G$ with $\sigma(-\mathbbm{1}_4)=0$, then the action 
\begin{align*}
\langle -\mathbbm{1}_4\rangle \times S\to S, ~\big(h,[(g,p)]\big)\mapsto [(hg,p)]
\end{align*}
leaves the bipolar immersion $\widetilde{\psi}\colon S\to\s^5$ invariant and induces a covering map of degree 2 on $S$
such that the corresponding quotient $S/\langle -\mathbbm{1}_4\rangle$ is orientable. In that case, we have
\begin{align*}
\textup{area}\Big(\widetilde{\mathcal{M}}_\Gamma\Big)\le \textup{area}(\mathcal{M}_\Gamma)-\pi\chi(S)\,.
\end{align*}
\item Analogously, if $S$ is non-orientable and $-\mathbbm{1}_{\R^4}\in G$, then the action
\begin{align*}
\langle -\mathbbm{1}_{4}\rangle \times \ov{S}\to \ov{S}, ~\big(h,[(s,g,p)]\big)\mapsto [(s,hg,p)]
\end{align*}
leaves the bipolar immersion $\widetilde{\psi}\colon\ov{S}\to\s^5$ invariant and induces a covering map of degree 2 on $\ov{S}$ such that the corresponding quotient $\ov{S}/\langle -\mathbbm{1}_4\rangle$ is orientable. In that case, we have
\begin{align*}
\textup{area}\Big(\widetilde{\mathcal{M}}_\Gamma\Big)\le 2\textup{area}(\mathcal{M}_\Gamma)-2\pi\chi(S)\,.
\end{align*}
\end{enumerate}
\end{theorem}
\begin{proof}
In each of the considered cases, the two-element group $\langle -\mathbbm{1}_{\R^4}\rangle$ acts smoothly, properly and freely on the manifolds $S$ and $\ov{S}$. Recalling the definition of the subgroup $G^p\subseteq G$ and Remark \ref{coveringmapGamma}, we note that $\langle -\mathbbm{1}_{\R^4}\rangle\cap G^p=\{\mathbbm{1}_{\R^4}\}$ for all $p\in \Delta$. Therefore, the corresponding orbit space projections yield smooth covering maps of degree two.
Moreover, the bilinearity of the wedge product yields that $\widetilde{\psi}$ is in both cases invariant under the respective action. 
The same holds, by the definition of the Gauss maps $\nu$ and $\ov{\nu}$, for the orientation forms $\omega$ and $\ov{\omega}$, implying that the quotient manifolds $S/\langle -\mathbbm{1}_{\R^4}\rangle$ and $\ov{S}/\langle -\mathbbm{1}_{\R^4}\rangle$ are orientable.
At last, the area estimates immediately result from Lemma \ref{bipolar_surface_conformal}.
\end{proof}

\section{Lawson's Bipolar $\widetilde{\xi}$- and $\widetilde{\eta}$-Surfaces}\label{xi_eta_surfaces}
Besides a general formulation, \cite{Lawson} presents a concrete application of the construction procedure for minimal surfaces in $\s^3$. 
The basic setting is described by positive integers $m$ and $k$ which specify a tessellation of $\mathbb{S}^3$ by congruent, geodesic tetrahedra. More detailed, the corresponding 1-skeleton is given by the set of great circles connecting equidistant points $Q_0, ..., Q_{2m-1}$ and $P_0,...,P_{2k-1}$ on two polar great circles. On that kind of lattice, Lawson introduced three distinct types of geodesic 4-gons satisfying the requirements to bound the initial piece of a closed minimal surface mentioned in the last section. In this way, he obtained the three families $(\xi_{m-1,k-1})$, $(\tau_{m,k})$ and $(\eta_{m-1,k-1})$ (note that the indices, depending on $m$ and $k$, are denoted such that the respective surface is based on a distance $\frac{\pi}{m}$ between $Q_j$ and $Q_{j+1}$, and $\frac{\pi}{k}$ between $P_i$ and $P_{j+1}$). The key point about these families is their different behavior concerning topology and embeddedness. At first, the surfaces $\xi_{m-1,k-1}$ yield examples of embedded and therefore orientable minimal surfaces of arbitrary genus $g>0$. Second, the surfaces $\tau_{m,k}$ are a family of immersed tori and Klein bottles, particularly including the Clifford torus $\tau_{1,1}$ (also realized by $\xi_{1,1}$). At last, the surfaces $\eta_{m-1,k-1}$ provide examples of non-orientable minimal surfaces for each genus, except for the real projective plane $\R P^2$ (in this context, also note that $\eta_{1,1}=\tau_{2,1}$, cf.\ \cite{Kusner}). For the $\tau$-family, \cite{Lawson} additionally provides corresponding parametrizations, which allow a very detailed analysis. For instance in \cite{Lapointe}, this is used for a topological classification of the corresponding bipolar surfaces $\widetilde{\tau}_{m,k}$. For a similar characterization of the other two families, we will now exploit the constructions and results from Section \ref{preimage_construction}. In the following let $m,k\in \Z_{\ge 2}$ and assume that $m>2$ or $k>2$ such that the Euler characteristics
\begin{align*}
\chi(\xi_{m-1,k-1})&=2\big(1-(m-1)(k-1)\big)\,,\\[5pt]
\chi(\eta_{m-1,k-1})&=\begin{cases}
1-(m-1)(k-1),&\textup{ if $k$ is even}\,;\\
2\big(1-(m-1)(k-1)\big),&\textup{ if $k$ is odd}
\end{cases}
\end{align*}
are negative, excluding the cases of $\xi_{1,1}=\tau_{1,1}$ and $\eta_{1,1}=\tau_{2,1}$ which were treated in \cite{Lapointe}. Let
\begin{align*}
P_i:=\begin{pmatrix}
\cos(i\cdot\nicefrac{\pi}{k})\\
\sin(i\cdot\nicefrac{\pi}{k})\\
0\\
0
\end{pmatrix},~ i\in\Z_{2k}\qquad,\qquad
Q_j:=\begin{pmatrix}
0\\
0\\
\cos(j\cdot\nicefrac{\pi}{m})\\
\sin(j\cdot\nicefrac{\pi}{m})\\
\end{pmatrix},~j\in\Z_{2m}
\end{align*}
be the points describing the tessellation of $\mathbb{S}^3$ used in \cite{Lawson}. In these terms, the geodesic polygon describing the orientable surface $\xi_{m-1,k-1}$ is given by the circuit
\begin{align}
\Gamma_{\xi_{m-1,k-1}}&:= P_0 Q_0 P_1 Q_1\,, \label{Gamma_xi}
\end{align}
where the successive vertices are connected by shortest arcs. Towards the corresponding group generated by Schwarz reflections, first note that the geodesic reflection $r_\gamma:\mathbb{S}^3\to\mathbb{S}^3$ across a great circle $\gamma$ in $\mathbb{S}^3\subseteq\R^4$ is the reflection at the 2-plane $P_\gamma$ such that $P_\gamma\cap \mathbb{S}^3=\gamma$, i.e.,
\begin{align*}
r_\gamma(x)=x^{||}-2 x^{\perp}
\end{align*}
for all $x=x^{||}+x^{\perp}\in\mathbb{S}^3$, where $x^{||}\in P_\gamma$ and $x^{\perp}\in P_\gamma^{\perp}$. Thereby, we determine the \textit{geodesic reflection $r_{ij}$ across the great circle $\gamma_{ij}$ through the points $P_i$ and $Q_j$}. We have
\begin{align*}
r_{00}=\begin{pmatrix}
\textbf{J}_2 &\textbf{0}\\
\textbf{0} & \textbf{J}_2
\end{pmatrix},
\quad
\textbf{J}_2:=\begin{pmatrix}
1 & \,~0\\
0 &-1
\end{pmatrix}
\end{align*}
and moreover, setting
\begin{align*}
\textbf{R}_\varphi:=\begin{pmatrix}
\cos(\varphi) & -\sin(\varphi)\\
\sin(\varphi) & \phantom{-}\cos(\varphi)
\end{pmatrix},~\varphi\in\R\,,
\end{align*}
as well as
\begin{align*}
\textbf{R}_\varphi^{\scriptscriptstyle (12)}:=\begin{pmatrix}
\textbf{R}_\varphi & \textbf{0}\\
\textbf{0}& \mathbbm{1}_2
\end{pmatrix},~\textbf{R}_\varphi^{\scriptscriptstyle (34)}:=\begin{pmatrix}
\mathbbm{1}_2& \textbf{0}\\
\textbf{0} & \textbf{R}_\varphi 
\end{pmatrix},
\end{align*}
we obtain
\begin{align*}
r_{ij}=r_{00}\cdot \textbf{R}_{\frac{2\pi}{k}\cdot i}^{\scriptscriptstyle (12)} \cdot \textbf{R}_{\frac{2\pi}{m}\cdot j}^{\scriptscriptstyle (34)}\,.
\end{align*}
Now, looking at \eqref{Gamma_xi}, the group generated by Schwarz reflections of the surface $\xi_{m-1, k-1}$ is, in the first place, given by
\begin{align}
G_{\xi_{m-1,k-1}}=\langle r_{00}, r_{01}, r_{11}, r_{10}\rangle\,.\label{G_xi_def}
\end{align}
Using 
\begin{align*}
r_{01}&=r_{00}\cdot \textbf{R}_{\frac{2\pi}{m}}^{\scriptscriptstyle (34)},\\
r_{11}&=r_{00}\cdot \textbf{R}_{\frac{2\pi}{k}}^{\scriptscriptstyle (12)}\cdot \textbf{R}_{\frac{2\pi}{m}}^{\scriptscriptstyle (34)}, \\
r_{10}&=r_{00}\cdot \textbf{R}_{\frac{2\pi}{k}}^{\scriptscriptstyle (12)}
\end{align*}
and conversely, since $r_{00}$ is self-inverse,
\begin{align}
\begin{aligned}
\textbf{R}_{\frac{2\pi}{k}}^{\scriptscriptstyle (12)}&=r_{00}\cdot r_{10}\\
\textbf{R}_{\frac{2\pi}{m}}^{\scriptscriptstyle (34)}&=r_{00}\cdot r_{01}\,,
\end{aligned}\label{rotations}
\end{align}
we find that
\begin{align*}
G_{\xi_{m-1,k-1}}&=\big\langle \textbf{R}_{\frac{2\pi}{k}}^{\scriptscriptstyle (12)},\textbf{R}_{\frac{2\pi}{m}}^{\scriptscriptstyle (34)}, r_{00}\big\rangle\,.
\end{align*}
As the block matrices $\textbf{R}_{\frac{2\pi}{k}}^{\scriptscriptstyle (12)}$ and $\textbf{R}_{\frac{2\pi}{m}}^{\scriptscriptstyle (34)}$ commute and in addition
\begin{align*}
\textbf{J}_2\cdot \textbf{R}_\varphi=\textbf{R}_\varphi^{-1}\cdot\textbf{J}_2
\end{align*} 
for all $\varphi\in\R$, this finally writes as
\begin{align}
G_{\xi_{m-1,k-1}}=\Bigg\{\Big(\textbf{R}_{\frac{2\pi}{k}}^{\scriptscriptstyle (12)}\Big)^\alpha\cdot \Big(\textbf{R}_{\frac{2\pi}{m}}^{\scriptscriptstyle (34)}\Big)^\beta\cdot r_{00}^\gamma:\alpha\in\Z_k,~\beta\in\Z_m,~\gamma\in\Z_2\Bigg\}\,.\label{G_xi}
\end{align}
For the parity $\sigma(g)$ of a group element
\begin{align*}
g=\Big(\textbf{R}_{\frac{2\pi}{k}}^{\scriptscriptstyle (12)}\Big)^\alpha\cdot \Big(\textbf{R}_{\frac{2\pi}{m}}^{\scriptscriptstyle (34)}\Big)^\beta\cdot r_{00}^\gamma\in G_{\xi_{m-1,k-1}}
\end{align*}
(cf.\ Definition \ref{parity}), we have that 
\begin{align}
\sigma(g)=\gamma\,,\label{sigma_xi}
\end{align}
since by \eqref{rotations}, $\textbf{R}_{\frac{2\pi}{k}}^{\scriptscriptstyle (12)}$ and $\textbf{R}_{\frac{2\pi}{m}}^{\scriptscriptstyle (34)}$ are given by an even number of Schwarz reflections.

Moreover, note that the subgroup $(G_{\xi_{m-1,k-1}})^{\Gamma_{\xi_{m-1,k-1}}}\subseteq G_{\xi_{m-1,k-1}}$ leaving the initial polygon invariant (as a point set) is trivial. To see this, consider for example the part $P_0 Q_0 P_1$ of $\Gamma_{\xi_{m-1,k-1}}$ which determines the angle $\frac{\pi}{k}$ at $Q_0$. For $g\in G_{\xi_{m-1,k-1}}$ (denoted as above), this polygonal arc is mapped onto
\begin{align}
g(P_0 Q_0 P_1)&=\Bigg[\Big(\textbf{R}_{\frac{2\pi}{k}}^{\scriptscriptstyle (12)}\Big)^\alpha\cdot \Big(\textbf{R}_{\frac{2\pi}{m}}^{\scriptscriptstyle (34)}\Big)^\beta\cdot r_{00}^\gamma\Bigg](P_0 Q_0 P_1)\notag\\[10pt]
&=\begin{cases}
P_{2\alpha} Q_{2\beta} P_{2\alpha+1}\quad\textup{if}~\gamma=0\,;\\[5pt]
P_{2\alpha} Q_{2\beta} P_{2\alpha-1}\quad\textup{if}~\gamma=1\,,
\end{cases}
\label{H_xi}
\end{align}
Now, if $g\in (G_{\xi_{m-1,k-1}})^{\Gamma_{\xi_{m-1,k-1}}}$, the prescribed angle yields that the candidates for the image of $P_0 Q_0 P_1$ under $g$ are
\begin{align*}
P_0 Q_0 P_1,~ P_1 Q_0 P_0,~P_1 Q_1 P_0,~ P_0 Q_1 P_1\,. 
\end{align*}
But combined with \eqref{H_xi}, this implies that we must have $g=\mathbbm{1}_4$.

\bigskip 
At this point, the results from Section \ref{preimage_construction} lead to the following conclusions concerning the bipolar surface $\widetilde{\xi}_{m-1,k-1}$.
\begin{theorem}\label{theorem_xi}
Let $m,\,k\in \Z_{\ge 2}$ such that $m>2$ or $k>2$. Then, the bipolar surface $\tilde{\xi}_{m-1,k-1}$ is orientable. Moreover,
\begin{enumerate}[label=\normalfont{(\roman*)}]
\item if both $m$ and $k$ are even, the Euler characteristic is 
\begin{align*}
\chi\Big(\tilde{\xi}_{m-1,k-1}\Big)=1-(m-1)(k-1)
\end{align*}
and we have
\begin{align*}
2\pi \max\{m, k\}\le\textup{area}\Big(\widetilde{\xi}_{m-1,k-1}\Big)<2\pi(mk+k-m)\,;
\end{align*}
\item if $m$ or $k$ is odd, the Euler characteristic is 
\begin{align*}
\chi\Big(\tilde{\xi}_{m-1,k-1}\Big)=2\big(1-(m-1)(k-1)\big)
\end{align*}
and we have
\begin{align*}
4\pi \max\{m, k\}\le\textup{area}\Big(\widetilde{\xi}_{m-1,k-1}\Big)<4\pi(mk+k-m)\,.
\end{align*}
\end{enumerate}
\end{theorem}
\begin{proof}
We use the notation from Construction \ref{construction} and Construction \ref{normal_oriented}. For a simpler notation, we define the points
\begin{align*}
\hat{P}_i:=\begin{pmatrix}
-\sin(i\cdot\nicefrac{\pi}{k})\\
\phantom{-}\cos(i\cdot\nicefrac{\pi}{k})\\
0\\
0
\end{pmatrix},~ i\in\Z_{2k}\qquad,\qquad
\hat{Q}_j=\begin{pmatrix}
0\\
0\\
-\sin(j\cdot\nicefrac{\pi}{m})\\
\phantom{-}\cos(j\cdot\nicefrac{\pi}{m})\\
\end{pmatrix},~j\in\Z_{2m}\,.
\end{align*}
The initial Gauß map $n:\Delta\to\s^3$ can be chosen such that at the vertices of $\Gamma_{\xi_{m-1,k-1}}$ we have
\begin{align*}
n(p_0)=n\big(f^{-1}(P_0)\big)=\hat{P}_0,\quad n\big(f^{-1}(Q_0)\big)=-\hat{Q}_0,\quad n\big(f^{-1}(P_1)\big)=-\hat{P}_1,\quad n\big(f^{-1}(Q_1)\big)=\hat{Q}_1\,.
\end{align*}
Connecting these vertices by shortest arcs yields the so-called \textit{polar polygon} $\Gamma_{\xi_{m-1,k-1}}^*$ of $\Gamma_{\xi_{m-1,k-1}}$ (cf.\ \cite{Lawson}, Section 10), which can be understood to bound the initial piece of the corresponding polar variety $\xi_{m-1,k-1}^*$ (cf. Theorem \ref{polar_surface_conformal}).

As a first step, we determine the multiplicity of the point $P_0\wedge\hat{P}_0=e_1\wedge e_2$ in the image of the bipolar immersion $\widetilde{\psi}\colon S\to\s^5$. Let therefore $p_0\in\partial\Delta$ be the point such that $f(p_0)=P_0=e_1$ and $n(p_0)=\hat{P}_0=e_2$.
We are looking for all $\big[(g,p)\big]\in S$ such that
\begin{align*}
\widetilde{\psi}\Big(\big[(g,p)\big]\Big)=\widetilde{\psi}\Big(\big[(e,p_0)\big]\Big).
\end{align*}
By definition of the map $\widetilde\psi=\psi\wedge\nu$, these are all the points $\big[(g,p)\big]\in S$ such that
\begin{align*}
(-1)^{\sigma(g)}\cdot (g\circ f)(p)\wedge (g\circ n)(p)= e_1\wedge e_2
\end{align*}
or equivalently, by relabelling the group elements, 
\begin{align}
f(p)\wedge n(p)=(-1)^{\sigma(g)}\cdot g(e_1)\wedge g(e_2)\,.\label{eq1}
\end{align}
Making use of \eqref{G_xi} and \eqref{sigma_xi}, we observe that
\begin{align*}
g(e_1)\wedge g(e_2)=\begin{cases}
\phantom{-}e_1\wedge e_2\quad\textup{if }\sigma(g)=0\,;\\
-e_1\wedge e_2\quad\textup{if }\sigma(g)=1\,.
\end{cases}
\end{align*}
Hence, \eqref{eq1} reduces to
\begin{align}
f(p)\wedge n(p)=e_1\wedge e_2\,.\label{eq2}
\end{align}
Clearly, \eqref{eq2} is satisfied by each point $\big[(g,p_0)\big]$ for arbitrary $g\in G_{\xi_{m-1,k-1}}$, and these are already all solutions.
Because, due to the properties of the wedge product (cf.\ Section \ref{wedge_product}), given any solution $\big[(g,p)\big]$ of \eqref{eq2}, we must have
\begin{align*}
f(p)&=\phantom{-}\cos(\phi)\, e_1+\sin(\phi)\,e_2\,,\\
n(p)&=-\sin(\phi)\,e_1+\cos(\phi)\,e_2
\end{align*}
for some $\phi\in [0,2\pi)$. Looking at the initial pieces of surface described by $f$ and $n$, which are both embedded in the convex hull of $\Gamma_{\xi_{m-1,k-1}}$ or $\Gamma_{\xi_{m-1,k-1}}^*$ we find that this is only satisfied for $\phi=0$, i.e., at the point $p_0$. 

At this point, we can conclude that
\begin{align*}
\widetilde{\psi}^{-1}\Big(P_0\wedge\hat{P}_0\Big)=\Big\{\big[(g,p_0)\big]:g\in G_{\xi_{m-1,k-1}}\Big\}.
\end{align*}
Therefore, since
\begin{align}
\big(G_{\xi_{m-1,k-1}}\big)^{p_0}=\langle r_{00}, r_{01}\rangle=\Bigg\{\Big(\textbf{R}_{\frac{2\pi}{m}}^{\scriptscriptstyle (34)}\Big)^\beta\cdot r_{00}^\gamma:\beta\in\Z_m,~\gamma\in\Z_2\Bigg\}\,,\label{G^p_xi}
\end{align}
as defined in \eqref{G^p}, the multiplicity of $P_0\wedge\hat{P}_0$ is given by
\begin{align*}
\mu_{\widetilde{\psi}}\Big(P_0\wedge\hat{P}_0\Big)&=\Big|\Big\{\big[(g,p_0)\big]:g\in G_{\xi_{m-1,k-1}}\Big\}\Big|=\frac{|G_{\xi_{m-1,k-1}}|}{\Big|\big(G_{\xi_{m-1,k-1}}\big)^{p_0}\Big|}=\frac{2mk}{2m}=k\,.
\end{align*}
Analogous steps lead to the conclusion that at the image point $$Q_0\wedge \big(-\hat{Q}_0\big)=\widetilde{\psi}\Big(\big[(e,q_0)\big]\Big)$$ the  multiplicity is
\begin{align*}
\mu_{\widetilde{\psi}}\Big(Q_0\wedge \big(-\hat{Q}_0\big)\Big)&=\Big|\Big\{\big[(g,q_0)\big]:g\in G_{\xi_{m-1,k-1}}\Big\}\Big|=\frac{|G_{\xi_{m-1,k-1}}|}{\Big|\big(G_{\xi_{m-1,k-1}}\big)^{q_0}\Big|}=\frac{2mk}{2k}=m\,.
\end{align*}
The question is whether these higher multiplicities result from a covering map on $S$ under which $\widetilde{\psi}$ is invariant. Clarification is obtained by considering the associated tangent planes to the bipolar surface at $P_0\wedge\hat{P}_0$ or $Q_0\wedge \big(-\hat{Q}_0\big)$.

Before we continue, note that since we assumed that $\chi(\xi_{m-1,k-1})<0$, the polar variety $\xi_{m-1,k-1}^*$ has branching points (cf.\ Proposition \ref{branchpoints}). According to \cite{Lawson}, p.\ 361, these particularly occur at the vertices of $\Gamma_{\xi_{m-1,k-1}}$ with interior angles $<\frac{\pi}{2}$. At such points, the second term of the differential $\dup\widetilde{\psi}=\dup\psi\wedge\nu +\psi\wedge\dup \nu$ vanishes.

Now, suppose that $m>2$ and denote by $P_{p_0}^{\scriptscriptstyle (0)}$ the tangent plane to the bipolar surface at $\widetilde{\psi}\big(\big[(e,p_0)\big]\big)$. 
Since in that case $\nu$ has a branching point at $\big[(e,p_0)\big]$ and the tangent plane to $\xi_{m-1,k-1}$ at $\psi\big(\big[(e,p_0)\big]\big)$ is spanned by $e_3$ and $e_4$, we find that 
\begin{align*}
P_{p_0}^{\scriptscriptstyle (0)}=\textup{span}\big(e_2\wedge e_3, e_2\wedge e_4)\,.
\end{align*}
Consequently, recalling \eqref{G^p_xi}, the $k$ tangent planes at $\widetilde{\psi}\Big(\Big[\Big(\Big(\textbf{R}_{\frac{2\pi}{k}}^{\scriptscriptstyle (12)}\Big)^\alpha,p_0\Big)\Big]\Big)$ are
\begin{align*}
P_{p_0}^{\scriptscriptstyle (\alpha)}:=\textup{span}\Bigg(\Big(\textbf{R}_{\frac{2\pi}{k}}^{\scriptscriptstyle (12)}\Big)^\alpha(e_2)\wedge \Big(\textbf{R}_{\frac{2\pi}{k}}^{\scriptscriptstyle (12)}\Big)^\alpha(e_3), \Big(\textbf{R}_{\frac{2\pi}{k}}^{\scriptscriptstyle (12)}\Big)^\alpha(e_2)\wedge \Big(\textbf{R}_{\frac{2\pi}{k}}^{\scriptscriptstyle (12)}\Big)^\alpha(e_4)\Bigg),~\alpha\in \Z_k\,,
\end{align*}
and a computation shows that $P_{p_0}^{\scriptscriptstyle (\alpha)}=P_{p_0}^{\scriptscriptstyle (0)}$ if and only if $\alpha=0$ or, when $k$ is even, $\alpha=\frac{k}{2}$. 

Analogously, if $k>2$, the $m$ tangent planes to the bipolar surface at $Q_0\wedge \big(-\hat{Q}_0\big)=f(q_0)\wedge n(q_0)$ are given by
\begin{align*}
P_{q_0}^{\scriptscriptstyle (\beta)}:=\textup{span}\Bigg(\Big(\textbf{R}_{\frac{2\pi}{m}}^{\scriptscriptstyle (34)}\Big)^\beta(e_1)\wedge \Big(\textbf{R}_{\frac{2\pi}{m}}^{\scriptscriptstyle (34)}\Big)^\beta(e_4), \Big(\textbf{R}_{\frac{2\pi}{k}}^{\scriptscriptstyle (34)}\Big)^\beta(e_2)\wedge \Big(\textbf{R}_{\frac{2\pi}{m}}^{\scriptscriptstyle (34)}\Big)^\beta(e_4)\Bigg),~\beta\in \Z_m\,,
\end{align*}
with $P_{q_0}^{\scriptscriptstyle (\beta)}=P_{q_0}^{\scriptscriptstyle (0)}$ if and only if $\beta=0$ or, when $m$ is even, $\beta=\frac{m}{2}$.

Whenever $m$ or $k$ is odd, this shows that the bipolar surface has $\mu$ transversally intersecting tangent planes at some point of multiplicity $\mu$, implying that $S$ is a fundamental domain of $\widetilde{\psi}$, i.e.,
\begin{align*}
\chi\Big(\widetilde{\xi}_{m-1,k-1}\Big)=\chi(S)\,.
\end{align*}
If both $m$ and $k$ are even, we have 
\begin{align*}
-\mathbbm{1}_4=\big(R_{\nicefrac{2\pi}{k}}^{\scriptscriptstyle (12)}\big)^{\frac{k}{2}}\cdot\big(R_{\nicefrac{2\pi}{m}}^{\scriptscriptstyle (34)}\big)^{\frac{m}{2}}\in G_{\xi_{m-1,k-1}}
\end{align*}
and in particular
\begin{align*}
\sigma(-\mathbbm{1}_4)=0\,.
\end{align*}
In this light, the occurrence of the pairs of parallel tangent planes is exactly due to the covering map from Theorem \ref{-1_cover} (i). Since $m>2$ or $k>2$, at least for one of the considered image points the planes corresponding to distinct pairs intersect transversally. So we deduce that $S/\langle-\mathbbm{1}_4\rangle$ is a fundamental domain of $\widetilde{\psi}$ and we have
\begin{align*}
\chi\Big(\widetilde{\xi}_{m-1,k-1}\Big)=\chi\big(S/\langle-\mathbbm{1}_4\rangle\big)=\frac{\chi(S)}{2}\,.
\end{align*}
It remains to prove the area bounds. The lower bounds are obtained by the Li-Yau inequality (cf.\ Theorem 6 in \cite{LiYau}, combined with Proposition 1.2.3 from \cite{KuwertSchaetzle}) applied to the vertex points, as for example the ones we studied above. If both $m$ and $k$ are even, these are points of multiplicity $\frac{m}{2}$ and $\frac{k}{2}$. Otherwise, the detected multiplicities are $m$ and $k$. Furthermore, together with the findings from above, the upper bounds on the area are a direct result of Lemma \ref{bipolar_surface_conformal} and the area bounds for the surface $\xi_{m-1,k-1}$ from Proposition 3.2 in \cite{Kusner}, i.e.,
\begin{align*}
\textup{area}(\xi_{m-1,k-1})<4\pi k\,.
\end{align*}
\end{proof}
\noindent We proceed with the $\eta$-family. The surface $\eta_{m-1,k-1}$ is based on the geodesic polygon
\begin{align*}
\Gamma_{\eta_{m-1,k-1}}&:= Q_0 P_1 Q_1 [P_0] (-Q_1)\,, 
\end{align*}
where the notation $[\,\cdot\,]$ indicates the choice of the arc of length $\pi$. Thus, its group generated by Schwarz reflections is given by
\begin{align}
G_{\eta_{m-1,k-1}}=\langle r_{10}, r_{11}, r_{01}, r_Q\rangle\,,\label{G_eta_def}
\end{align}
where $r_Q$ denotes the geodesic reflection at $\gamma_Q:=\{x\in\mathbb{S}^3: x_1=x_2=0\}$. To study $G_{\eta_{m-1,k-1}}$ in more detail, we distinguish between the case of $k$ being even, when $\eta_{m-1, k-1}$ is non-orientable, and the case of $k$ being odd, when $\eta_{m-1, k-1}$ is orientable.
At first, let $k$ be even. Since
\begin{align}
(r_{01}\cdot r_{11})^\frac{k}{2}=\Big(\textbf{R}_{\frac{2\pi}{k}}^{\scriptscriptstyle (12)}\Big)^\frac{k}{2}=r_Q\,,\label{rQ}
\end{align}
$r_Q$ can be dropped as a generator. Referring to Proposition \ref{nonorientability}, this relation is equivalent to
\begin{align*}
\mathbbm{1}_4= (r_{11}\cdot r_{01})^\frac{k}{2}\cdot r_Q
\end{align*}
and hence shows that $\eta_{m-1,k-1}$ is non-orientable. Now, analogous to the case of the $\xi$-family, it follows that
\begin{align}
G_{\eta_{m-1,k-1}}
=\big\langle \textbf{R}_{\frac{2\pi}{k}}^{\scriptscriptstyle (12)},\textbf{R}_{\frac{2\pi}{m}}^{\scriptscriptstyle (34)}, r_{00}\big\rangle
=\Big\{\Big(\textbf{R}_{\frac{2\pi}{k}}^{\scriptscriptstyle (12)}\Big)^\alpha\cdot \Big(\textbf{R}_{\frac{2\pi}{m}}^{\scriptscriptstyle (34)}\Big)^\beta\cdot r_{00}^\gamma:\alpha\in\Z_k,~\beta\in\Z_m,~\gamma\in\Z_2\Big\}\,.\label{G_eta_nor}
\end{align}
Let now $k$ be odd. In this case, we cannot drop the generator $r_Q$. By the fact that $r_Q$ commutes with all the other generators from \eqref{G_eta_def}, we have
\begin{align*}
G_{\eta_{m-1,k-1}}&\cong \langle r_Q\rangle \times \langle r_{10}, r_{11}, r_{01}\rangle\,,
\end{align*}
i.e.,
\begin{align}
G_{\eta_{m-1,k-1}}
&=\Bigg\{r_Q^\gamma\cdot \Big(\textbf{R}_{\frac{2\pi}{k}}^{\scriptscriptstyle (12)}\Big)^\alpha\cdot \Big(\textbf{R}_{\frac{2\pi}{m}}^{\scriptscriptstyle (34)}\Big)^\beta\cdot r_{00}^\delta:\alpha\in\Z_k,~\beta\in\Z_m,~\gamma,~\delta\in\Z_2\Bigg\}\,.\label{G_eta_or}
\end{align}
For the parity $\sigma(g)$ of 
\begin{align*}
g=r_Q^\gamma\cdot\big(R_{\nicefrac{2\pi}{k}}^{\scriptscriptstyle (12)}\big)^\alpha\cdot \big(R_{\nicefrac{2\pi}{m}}^{\scriptscriptstyle (34)}\big)^\beta\cdot r_{00}^\delta \in G_{\eta_{m-1,k-1}}\,,
\end{align*}
we have
\begin{align}
\sigma(g)=\gamma+\delta\,.\label{sigma_eta}
\end{align}
Furthermore, note that independently of the parities of $m$ and $k$, we have that the subgroup of $\textup{O}(4)$ which leaves 
\begin{align*}
\Gamma_{\eta_{m-1,k-1}}= Q_0 P_1 Q_1 [P_0] (-Q_1) 
\end{align*}
invariant as a point set is trivial. To see this, one can first consider the possible images under a symmetry of $Q_1 [P_0] (-Q_1)$, the only arc of length $\pi$, and then images of the piece $Q_0 P_1 Q_1$ if $m>2$, or images of $P_1 Q_1 [P_0] (-Q_1)$ if $k>2$.

\bigskip
We finally arrive at the topological classification of the bipolar $\tilde{\eta}$-family.
\begin{theorem}\label{theorem_eta}
Let $m,\,k\in \Z_{\ge 2}$ such that $m>2$ or $k>2$. Then, the bipolar surface $\tilde{\eta}_{m-1,k-1}$ is orientable. Moreover,
\begin{enumerate}[label=\normalfont{(\roman*)}]
\item if both $m$ and $k$ are even, the Euler characteristic is 
\begin{align*}
\chi\Big(\tilde{\eta}_{m-1,k-1}\Big)=1-(m-1)(k-1)
\end{align*}
and we have
\begin{align*}
2\pi \max\{m, k\}\le\textup{area}\Big(\widetilde{\eta}_{m-1,k-1}\Big)<2\pi(3mk -3k - m)\,;
\end{align*}
\item if $m$ or $k$ is odd, the Euler characteristic is 
\begin{align*}
\chi\Big(\tilde{\eta}_{m-1,k-1}\Big)= 2\big(1-(m-1)(k-1)\big)
\end{align*}
and we have
\begin{align*}
4\pi \max\{m, k\}\le\textup{area}\Big(\widetilde{\eta}_{m-1,k-1}\Big)<4\pi(3mk -3k - m)\,.
\end{align*}
%
\end{enumerate}
\end{theorem}
\begin{proof}
The initial Gauß map $n:\Delta\to\s^3$ can be chosen such that we have
\begin{align*}
n\big(f^{-1}(Q_0)\big)=\hat{P}_1,\quad n\big(f^{-1}(P_1)\big)=-\hat{P}_1,\quad n\big(f^{-1}(Q_1)\big)=\hat{Q}_1,\quad n\big(f^{-1}(-Q_1)\big)=\hat{P}_0\,,
\end{align*}
at the vertices of $\Gamma_{\eta_{m-1,k-1}}$. Connected by shortest arcs, these values describe the polar polygon $\Gamma_{\eta_{m-1,k-1}}^*$ of $\Gamma_{\eta_{m-1,k-1}}$, where we note that the arc from $\hat{P}_1$ to $-\hat{P}_1$ runs across $-\hat{Q}_0$.

\bigskip
If $k$ is even, then $\eta_{m-1,k-1}$ is non-orientable and we use the notation from Construction \ref{doublecover}. In this setting, the multiplicity of an image point $\widetilde{\psi}\Big(\big[(0,e,p_0)\big]\Big)$ is determined by the numbers of solutions $\big[(s,g,p)\big]\in \widetilde{S}$ of
\begin{align}
f(p)\wedge n(p)=(-1)^s g\big(f(p_0)\big)\wedge g\big(n(p_0)\big)\,.\label{eq3}
\end{align}
Now, if $p_0\in\partial\Delta$ is such that $f(p_0)=P_1$ or $f(p_0)=Q_1$, then the characterization of the group from \eqref{G_eta_nor} implies that \eqref{eq3} is equivalent to
\begin{align*}
f(p_0)\wedge n(p_0)=(-1)^{s+\gamma} f(p_0)\wedge n(p_0)\,.
\end{align*}
Consequently, we have
\begin{align*}
\mu_{\widetilde{\psi}}\Bigg(\widetilde{\psi}\Big(\big[(0,e,p_0)\big]\Big)\Bigg)=\frac{1}{2}\cdot \frac{2\cdot |G_{\eta_{m-1,k-1}}|}{\Big|\big(G_{\eta_{m-1,k-1}}\big)^{p_0}\Big|}\,,
\end{align*}
i.e., 
\begin{align*}
\mu_{\widetilde{\psi}}\Big(P_1\wedge\big(-\hat{P}_1\big)\Big)&=k\,,\quad
\mu_{\widetilde{\psi}}\Big(Q_1\wedge Q_1\Big)=m\,.
\end{align*}
Then, as in the proof of Theorem \ref{theorem_xi}, considering the different tangent planes at these points of higher multiplicity
and additionally using Theorem \ref{-1_cover} (ii) if $m$ is even, yields that a fundamental domains of $\widetilde{\psi}$ is given by $\widetilde{S}$ if $m$ is odd, and by $\widetilde{S}/\langle -\mathbbm{1}_4\rangle$ if $m$ is even. 

\bigskip
In contrast, if $k$ is odd, then $\eta_{m-1,k-1}$ is orientable and we are in the setting of Construction \ref{construction} and Construction \ref{normal_oriented}. Considering the image point $P_1\wedge\big(-\hat{P}_1\big)=-e_1\wedge e_2$ and using \eqref{G_eta_or} as well as \eqref{sigma_eta}, it follows analogously as in the proof of Theorem \ref{theorem_xi} that $S$ is a fundamental domain for $\widetilde{\psi}$.   

Finally, the area bounds follow analogously as in the proof of Theorem \ref{theorem_xi} by the Li-Yau inequality, by the formula from Lemma \ref{bipolar_surface_conformal} (holding on the orientable double cover $\ov{S}$ in each of the non-orientable cases) and by Proposition 3.4 from \cite{Kusner}, i.e.,
\begin{align*}
\textup{area}(\eta_{m-1,k-1})<2\pi(m-1)k \quad\textup{if $k$ is even},
\end{align*}
which we completed by
\begin{align*}
\textup{area}(\eta_{m-1,k-1})<4\pi(m-1)k\quad\textup{if $k$ is odd}.
\end{align*}
The latter follows analogously as in \cite{Kusner}, i.e., by the bound for the initial minimal disk, multiplied by the order of the group generated by Schwarz reflections (which is twice the order of the former case).
\end{proof}

\bigskip
As both for the surfaces $\widetilde{\xi}_{m-1,k-1}$ and $\widetilde{\eta}_{m-1,k-1}$ we detected transversally intersecting tangent planes, the following is immediate.
\begin{corollary}\label{embeddedness}
Let $m,\,k\in \Z_{\ge 2}$ such that $m>2$ or $k>2$. Then, $\widetilde{\xi}_{m-1,k-1}$ and $\widetilde{\eta}_{m-1,k-1}$ are not embedded.
\end{corollary}

\bibliographystyle{plain} 
\bibliography{bibliography2.bib}

\begin{center}
Elena Mäder-Baumdicker at \\
maeder-baumdicker@mathematik.tu-darmstadt.de\\
Department of Mathematics, Technische Universität Darmstadt\\
Schlossgartenstraße 7, 64289 Darmstadt, Germany

\bigskip
Melanie Rothe at \\
rothe@mathematik.tu-darmstadt.de\\
Department of Mathematics, Technische Universität Darmstadt\\
Schlossgartenstraße 7, 64289 Darmstadt, Germany
\end{center}
\end{document}